\documentclass{siamart190516}

\usepackage{amsfonts}
\usepackage{graphicx}
\usepackage{epstopdf}
\usepackage{algorithmic}

\usepackage{amssymb}
\usepackage{lineno}
\usepackage{hyperref}
\usepackage{amsmath}
\usepackage{stmaryrd}
\usepackage{multirow}
\usepackage{enumitem}
\usepackage{dsfont}
\usepackage{makecell}
\usepackage{soul}
\usepackage{appendix}
\usepackage{multirow}
\usepackage{arydshln}
\setuldepth{Paris}

\usepackage{amssymb}
\usepackage{url}
\usepackage{subcaption}
\usepackage[all,cmtip]{xy}
\usepackage{mathtools}
\usepackage{yhmath}

\usepackage{lipsum}
\usepackage{enumitem}
\setlist{nosep}

\ifpdf
  \DeclareGraphicsExtensions{.eps,.pdf,.png,.jpg}
\else
  \DeclareGraphicsExtensions{.eps}
\fi

\newenvironment{NDLR}{\em}{}

\newcommand{\mc}{\mathcal}

\newcommand{\ls}{\leqslant}
\newcommand{\gs}{\geqslant}
\newcommand{\mini}{\mathrm{min}}
\newcommand{\maxi}{\mathrm{max}}

\newcommand{\N}{\mathbb{N}}

\newcommand{\R}{\mathbb{R}}

\newcommand{\Sph}{\mathbb{S}}

\newcommand{\lto}{\longrightarrow}
\newcommand{\lmto}{\longmapsto}

\newcommand{\nec}{\Longrightarrow}
\newcommand{\cns}{\Longleftrightarrow}

\newcommand{\im}{\mathrm{im\,}}

\newcommand{\Span}{\mathrm{span}}

\newcommand{\tr}{\mathrm{tr}}

\newcommand{\rk}{\mathrm{rk}}

\newcommand{\sym}{\mathrm{sym}}

\renewcommand{\sp}{\mathrm{sp}}

\newcommand{\Diag}{\mathrm{Diag}}
\newcommand{\Mat}{\mathrm{Mat}}
\newcommand{\Sym}{\mathrm{Sym}}
\newcommand{\Skew}{\mathrm{Skew}}

\newcommand{\GL}{\mathrm{GL}}
\newcommand{\Orth}{\mathrm{O}}

\newcommand{\Cov}{\mathrm{Cov}}

\newcommand{\St}{\mathrm{St}}

\newcommand{\pow}{\mathrm{pow}}

\newcommand{\Stab}{\mathrm{Stab}}

\newcommand{\Orb}{\mathrm{Orb}}

\newcommand{\Exp}{\mathrm{Exp}}
\newcommand{\Log}{\mathrm{Log}}

\newcommand{\fun}[4]{
\left\{
\begin{array}{ccc}
#1 & \lto & #2\\ \relax
#3 & \lmto & #4\\
\end{array}
\right.
}

\newcommand{\chixlemmaeucbwtoptitle}{Euclidean and Bures-Wasserstein topologies coincide}
\newcommand{\chixlemmaeucbwtop}{The Euclidean distance $d^{\mathrm{E}}$ and the Bures-Wasserstein distance $d^{\mathrm{BW}}$ define the same topology on $\Cov(n)$.}

\newcommand{\chixthmbwgeodsymntitle}{Bures-Wasserstein geodesics on $\Sym^+(n)$}

\newcommand{\chixthmhorliftmettitle}{Horizontal lift, tangent space, metric}
\newcommand{\chixthmhorliftmet}{Let $\Sigma\in\Sym^+(n,k)$, let $X\in\R^{n\times k}_*$ such that $\Sigma=XX^\top\in\Sym^+(n,k)$ and let $V\in T_\Sigma\Sym^+(n,k)$. Let $\Sigma=UDU^\top$ be a singular value decomposition with $D\in\Diag^+(k)$ and $U\in\St(n,k)$. We denote $S=S_{\Sigma,V}=U\mc{S}_D(U^\top VU)U^\top$, where $\mc{S}_A(B)$ denotes the unique solution $Z$ of the Sylvester equation $AZ+ZA=B$. Note that $S_{\Sigma,V}$ and $(I_n-UU^\top)$ are independent from the chosen decomposition.
\begin{enumerate}
    \item (Tangent space) $T_\Sigma\Sym^+(n,k)=\{V\in\Sym(n)|X_\perp^\top VX_\perp=0\}$,
    \item (Horizontal lift) $V^\#_X=X(X^\top X)^{-1}\mc{S}_{X^\top X}(X^\top VX)+X_\perp X_\perp^\top VX(X^\top X)^{-1}$, where $X_\perp X_\perp^\top=I_n-X(X^\top X)^{-1}X^\top$,
    \item (Bures-Wasserstein metric) $g_\Sigma^{BW(n,k)}(V,V)=\tr(S_{\Sigma,V}\Sigma S_{\Sigma,V}+V\Sigma^- V(I_n-UU^\top))$.
\end{enumerate}}

\newcommand{\chixthmbwgeodsymnktitle}{Bures-Wasserstein geodesics on $\Sym^+(n,k)$}

\newcommand{\chixthmmingeodtitle}{Bures-Wasserstein minimizing geodesics in $\Cov(n)$}
\newcommand{\chixthmmingeod}{Let $\Sigma,\Lambda\in\Cov(n)$ with $\rk(\Sigma)=k$ and $\rk(\Lambda)=l$. Let $X,Y\in\R^{n\times n}$ such that $XX^\top=\Sigma$ and $YY^\top=\Lambda$. The two following statements are equivalent:
\begin{enumerate}[label=(\roman*)]
    \item the curve $\gamma:[0,1]\lto\Cov(n)$ is a minimizing geodesic segment from $\Sigma$ to $\Lambda$,
    \item there exists $R\in\Orth(n)$ such that $H_{X,Y,R}:=X^\top YR^\top\in\Cov(n)$ and for all $t\in[0,1]$, $\gamma(t)=\gamma_{\Sigma\to\Lambda}^R(t)=(1-t)^2\Sigma+t^2\Lambda+2t(1-t)\sym(XRY^\top)$.
\end{enumerate}
Moreover, $H_{X,Y,R}=(X^\top\Lambda X)^{1/2}$ and the minimizing geodesic $\gamma_{\Sigma\to\Lambda}^R$ is of constant rank $p\gs\max(k,l)$ on $(0,1)$.}

\newcommand{\chixlemmaelemalgtitle}{Elementary algebra}
\newcommand{\chixlemmaelemalg}{Let $\Sigma,\Lambda\in\Cov(n)$ with $\rk(\Sigma)=k$ and $\rk(\Lambda)=l$. Let $r=\rk(\Sigma\Lambda)$.
\begin{enumerate}
    \itemsep0em
    \item For all $X,Y\in\Mat(n)$ such that $XX^\top=\Sigma$ and $YY^\top=\Lambda$, $r=\rk(X^\top Y)$.
    \item We have $l-r\ls n-k$.
\end{enumerate}}

\newcommand{\chixthmnbbwgeodtitle}{Number of Bures-Wasserstein minimizing geodesic segments in $\Cov(n)$}

\newsiamremark{remark}{Remark}
\newsiamremark{hypothesis}{Hypothesis}
\crefname{hypothesis}{Hypothesis}{Hypotheses}
\newsiamthm{claim}{Claim}
\newsiamthm{conjecture}{Conjecture}

\headers{Bures-Wasserstein geodesics}{Y. Thanwerdas and X. Pennec}

\title{Bures-Wasserstein minimizing geodesics between covariance matrices of different ranks \thanks{Preprint, April 1st, 2022.}}

\author{Yann Thanwerdas\thanks{Université Côte d'Azur and Inria, Epione Project Team (\email{yann.thanwerdas@inria.fr}).}
\and Xavier Pennec\thanks{Université Côte d'Azur and Inria, Epione Project Team (\email{xavier.pennec@inria.fr}).}}

\usepackage{amsopn}

\ifpdf
\hypersetup{
  pdftitle={Bures-Wasserstein geodesics},
  pdfauthor={Y. Thanwerdas and X. Pennec}
}
\fi

\begin{document}

\maketitle

\begin{abstract}
The set of covariance matrices equipped with the Bures-Wasserstein distance is the orbit space of the smooth, proper and isometric action of the orthogonal group on the Euclidean space of square matrices. This construction induces a natural orbit stratification on covariance matrices, which is exactly the stratification by the rank. Thus, the strata are the manifolds of symmetric positive semi-definite (PSD) matrices of fixed rank endowed with the Bures-Wasserstein Riemannian metric. In this work, we study the geodesics of the Bures-Wasserstein distance. Firstly, we complete the literature on geodesics in each stratum by clarifying the set of preimages of the exponential map and by specifying the injection domain. We also give explicit formulae of the horizontal lift, the exponential map and the Riemannian logarithms that were kept implicit in previous works. Secondly, we give the expression of all the minimizing geodesic segments joining two covariance matrices of any rank. More precisely, we show that the set of all minimizing geodesics between two covariance matrices $\Sigma$ and $\Lambda$ is parametrized by the closed unit ball of $\R^{(k-r)\times(l-r)}$ for the spectral norm, where $k,l,r$ are the respective ranks of $\Sigma,\Lambda,\Sigma\Lambda$. In particular, the minimizing geodesic is unique if and only if $r=\min(k,l)$. Otherwise, there are infinitely many. 
\end{abstract}

\begin{keywords}
Covariance matrices, PSD matrices, Bures-Wasserstein, orbit space, geodesics, injection domain
\end{keywords}

\begin{AMS}
15B48, 15A63, 53B20, 53C22, 58D17, 53-08, 53A04, 54E50, 58A35.
\end{AMS}

\section{Introduction}\label{sec:bures-wasserstein_introduction}

Many data can be represented as covariance matrices. They are often assumed to be Symmetric Positive Definite (SPD) because it is much more convenient from the geometric point of view. Indeed, the set of SPD matrices is an open convex cone in the vector space of symmetric matrices so it has a canonical differential structure. The induced Euclidean metric is not satisfying to compute with SPD matrices because geodesics leave the space in finite time and interpolations are often non-realistic. To solve this problem, a lot of Riemannian metrics were proposed on SPD matrices, mainly $\Orth(n)$-invariant metrics \cite{Thanwerdas22-LAA} (affine-invariant \cite{Siegel43,Skovgaard84,Pennec06,Lenglet06-JMIV,Fletcher07,Moakher05}, log-Euclidean \cite{Arsigny06,Fillard07,HaQuang14}, Bures-Wasserstein \cite{Dowson82,Olkin82,Takatsu10,Takatsu11,Malago18,Bhatia19}, Bogoliubov-Kubo-Mori \cite{Petz93,Michor00}, etc.), Cholesky-like metrics \cite{Li17,Lin19} or product metrics with one metric on positive diagonal matrices and one metric on full-rank correlation matrices \cite{Thanwerdas21-GSI,Thanwerdas22-SIMAX}.

However, this viewpoint often forgets about singular covariance matrices, that is covariance matrices with non-full rank. Altogether, they form a closed convex cone which is not anymore a manifold. First, it can be equipped with distances to provide a metric space structure. The Euclidean distance is not satisfying either here because geodesics leave the closed cone in finite time again. The main alternative is the Bures-Wasserstein distance that was introduced in many different ways. For example in \cite{Dowson82,Olkin82}, it is the Wasserstein/Kantorovitch distance between multivariate centered Gaussian distributions, possibly degenerate. It is also the quotient distance of the Euclidean metric on square matrices by the right action of the orthogonal group. This is why it was also called the Procrustes distance \cite{Dryden09,HaQuang22}. This viewpoint allows to split the closed cone into strata that are Riemannian manifolds whose induced geodesic distance is precisely the Bures-Wasserstein distance. In particular, the space of covariance matrices equipped with this distance is a complete geodesic metric space.

The geometry of stratified spaces is a topic of interest in the community of statistics in non-linear spaces. Examples of popular stratified spaces are the Kendall shape spaces \cite{Kendall84}, the BHV space of trees \cite{Billera01}, the QED space of trees \cite{Feragen11}, the Graph space \cite{Calissano20}, the Wald space of forests \cite{Garba21}, the correlation matrices or the symmetric/diagonal matrices stratified by eigenvalue multiplicity. Moreover, the space of covariance matrices with the Bures-Wasserstein distance is a metric space of non-negative curvature \cite{Takatsu11}. Spaces of this type have been much less described than metric spaces of non-positive curvature \cite{Bridson99}. These are two motivations to study the Bures-Wasserstein geometry of covariance matrices. In this work, we focus on geodesics.

We would like to highlight three important elements that depart from usual studies on the geometry of covariance matrices. Firstly, contrarily to the affine-invariant, the log-Euclidean or the log-Cholesky metrics, the Bures-Wasserstein metric on SPD matrices (or on PSD matrices of fixed rank) is not geodesically complete. Therefore in this work, we need the notion of \textit{definition domain} of the exponential map because it is not $\R$ in general. Secondly, in the cited examples, the geodesics are globally minimizing on $\R$. It is not the case for the Bures-Wasserstein metric so we need the notions of \textit{cut time} and \textit{injection domain} to specify when the geodesic stops to be minimizing. Thirdly, in this work we talk about \textit{geodesics in a metric space}, not only in a Riemannian manifold as previously. We recall the definition of these notions in Section \ref{sec:bures-wasserstein_concepts}.

In the principal (or regular) stratum of SPD matrices, the Bures-Wasserstein Riemannian metric was extensively studied. Since $X\in\GL(n)\lmto XX^\top\in\Sym^+(n)$ is a Riemannian submersion, many geometric operations can be computed thanks to O'Neill's equations \cite{ONeill66}. Therefore, the curvature was derived in \cite{Takatsu10,Takatsu11,Thanwerdas22-LAA}, the quotient geometry was described in \cite{Malago18,Bhatia19,Oostrum20}, the exponential map was computed in \cite{Malago18}, a Riemannian logarithm was given in \cite{Bhatia19}, the injectivity radius was computed in \cite{Massart20} and a simplified equation of the geodesic parallel transport was proposed in \cite{Thanwerdas22-LAA}. In contrast to geodesically complete Riemannian metrics, it is important to specify the definition domain of the exponential map. It was characterized in \cite{Malago18} as the connected component of 0 in a subset of $\R$, which could be specified more explicitly. Moreover, the uniqueness of the Riemannian logarithm is not established and the injection domain seems to be unknown. Therefore in Section \ref{sec:bures-wasserstein_sym^+(n)}, we clarify the definition domain of the exponential map, we prove the uniqueness of the logarithm thanks to a result from \cite{Massart20} and we prove that the geodesics are minimizing on their domain of definition, which also provides the injection domain.

In each other stratum of PSD matrices of fixed rank $k<n$, the Bures-Wasserstein Riemannian metric was studied via the analogous Riemannian submersion defined by $X\in\R^{n\times k}_*\lmto XX^\top\in\Sym^+(n,k)$, where $\R^{n\times k}_*$ is the open set of matrices of full rank $k$ in $\R^{n\times k}$ and $\Sym^+(n,k)$ is the set of PSD matrices of size $n$ and rank $k$. The curvature was computed in \cite{Massart19}, the exponential map, its domain of definition, the logarithm map and the injectivity radius were derived in \cite{Massart20}. The horizontal lift was kept implicit so these results are formulated in the total space $\R^{n\times k}_*$. We think that it is easier to understand the geometry with formulae depending directly on the tangent vector and not on its horizontal lift. Therefore in Section \ref{sec:bures-wasserstein_sym^+(n,k)}, we compute the horizontal lift and we give the expressions of the Riemannian metric, the exponential map and its definition domain in function of vectors tangent to the manifold $\Sym^+(n,k)$. Moreover, we notice that the solutions of the logarithm equation characterized in \cite{Massart20} may not be in the definition domain. Thus we explain the additional condition for being a preimage of the exponential map. In addition, we give an explicit bijective parametrization of the Riemannian logarithms, which allows us to count them. When it is unique, we give an explicit formula of the corresponding minimizing geodesic in function of the end points. This finally allows us to compute the injection domain which is kept implicit in \cite{Massart20}.

Beyond the clarification and completion of the literature on geodesics in each stratum, our main objective is to characterize the minimizing geodesics \textit{between} strata. Our main results are the following. (1) Any minimizing geodesic segment between two covariance matrices $\Sigma$ and $\Lambda$ is of constant rank on the interior of the segment. It is called the rank of the minimizing geodesic and it is greater than the ranks of $\Sigma$ and $\Lambda$. (2) We give the explicit formula of all the minimizing geodesic segments in Theorem \ref{thm:minimizing_geodesics}. (3) They are parametrized by the 
vectors of the closed unit ball of $\R^{(k-r)\times(l-r)}$ for the spectral norm, where $k,l,r$ are the respective ranks of $\Sigma,\Lambda,\Sigma\Lambda$. In other words, they are parametrized by matrices $R_0\in\R^{(k-r)\times(l-r)}$ with singular values in $[0,1]$. (4) The minimizing geodesic segment is unique if and only if $r=\min(k,l)$ (this includes $\max(k,l)=n$). Otherwise, there are infinitely many. (5) The number of minimizing geodesics of minimal rank (i.e. of rank equal to $\max(k,l)$) is 1 if $r=\min(k,l)$, otherwise it is 2 if $k=l$, otherwise it is infinite. (6) Assuming $k\gs l$, if $R_0$ belongs to the Stiefel manifold $\St(k-r,l-r)$, that is $R_0^\top R_0=I_{l-r}$, then the corresponding geodesic is of minimal rank. (7) The choice of parameter $R_0=0$ leads to the geodesic $\gamma_{\Sigma\to\Lambda}^0(t)=(1-t)^2\Sigma+t^2\Lambda+2t(1-t)\sym(\Sigma^{1/2}((\Sigma^{1/2}\Lambda\Sigma^{1/2})^{1/2})^-\Sigma^{1/2}\Lambda)$ whose expression does not depend on the ranks of $\Sigma$ and $\Lambda$. It is called the Bures-Wasserstein canonical geodesic.

In the remainder of this section, we introduce some matrix notations. In Section \ref{sec:bures-wasserstein_concepts}, we introduce the important concepts of geodesics and quotient space in metric spaces and manifolds. We give a particular attention to the notions that characterize if a geodesic (self-parallel curve) is minimizing: cut time, injection domain, difference between preimages of the exponential map and Riemannian logarithms. In Section \ref{sec:bures-wasserstein_geometry}, we recall the algebraic structure, metric topology and differential geometry of the convex cone of covariance matrices seen as the quotient of square matrices by the orthogonal group. In Sections \ref{sec:bures-wasserstein_sym^+(n)} and \ref{sec:bures-wasserstein_sym^+(n,k)}, we complete the literature on the geodesics of the Bures-Wasserstein metric on SPD matrices and on singular matrices of fixed rank respectively. In Section \ref{sec:bures-wasserstein_cov(n)}, we give our main results on minimizing geodesics in the whole Bures-Wasserstein metric space of covariance matrices. We conclude in Section \ref{sec:bures-wasserstein_conclusion}. The main proofs are deferred to the Appendix.\\

\paragraph{Matrix notations}

Let $n,k\in\N$. In this work, we use the following manifolds of matrices:
\begin{enumerate}[label=$\cdot$]
    \itemsep0em
    \item the vector space of $n\times k$ matrices $\R^{n\times k}$,
    \item the open subset $\R^{n\times k}_*\subset\R^{n\times k}$ of full-rank matrices,
    \item in particular, the vector space of square matrices $\Mat(n)=\R^{n\times n}$ and the general linear group $\GL(n)=\R^{n\times n}_*$,
    \item the orthogonal group $\Orth(n)$,
    \item the Stiefel manifold $\St(n,k)=\Orth(n)/\Orth(n-k)$,
    \item the manifold of symmetric positive definite matrices $\Sym^+(n)$,
    \item the manifold of symmetric positive semi-definite matrices of fixed rank $k$, $\Sym^+(n,k)$,
    \item the vector space of diagonal matrices $\Diag(n)$,
    \item the groups of invertible diagonal matrices $\Diag^*(n)=\Diag(n)\cap\GL(n)$ and positive diagonal matrices $\Diag^+(n)=\Diag(n)\cap\Sym^+(n)$.\\
\end{enumerate}

In Section \ref{sec:bures-wasserstein_cov(n)}, we need to distinguish cases where matrices may have one or two null dimensions, i.e. belonging to $\R^{n\times 0}$, $\R^{0\times k}$ or $\R^{0\times 0}$. Thus we recall that these spaces are isomorphic to the vector space $\{0\}$. Indeed, there is a unique linear map from $\R^n$ to $\R^0$ or from $\R^0$ to $\R^k$, which is the identically null map. The canonical basis of $\R^0$ is empty and the corresponding matrix in the canonical bases is called the empty matrix. It is practical to treat these spaces as non-trivial spaces to avoid writing particular cases. In particular, $\St(n,0)=\R^{n\times 0}$ and $\Orth(0)=\GL(0)=\Mat(0)=\Diag(0)$ are sets of cardinal 1.\\

We use the following notations.
\begin{enumerate}[label=$\cdot$]
    \itemsep0em
    \item $I_n$ denotes the identity matrix of size $n$.
    \item $\mathbf{0}_n$ denotes the null matrix of size $n\times n$. $\mathbf{0}_{n,k}$ denotes the null matrix of size $n\times k$. We may simply denote them $0$ when sizes are obvious in the context.
    \item (Sylvester equation) $\mc{S}_A(B)$ is the unique solution $Z$ of the Sylvester equation $AZ+ZA=B$ for $A\in\Sym^+(k)$ and $B\in\Sym(k)$.
    \item (Löwner order) For all $\Sigma\in\Sym(n)$, we say that $\Sigma$ is positive definite (resp. positive semi-definite) and we denote $\Sigma>0$ (resp. $\Sigma\gs 0$) when $\Sigma$ has positive (resp. non-negative) eigenvalues. Given $\Lambda\in\Sym(n)$, we denote $\Sigma>\Lambda$ (resp. $\Sigma\gs\Lambda$) when $\Sigma-\Lambda>0$ (resp. $\Sigma-\Lambda\gs 0$).\\
\end{enumerate}

We recall basic facts on symmetric matrices.
\begin{enumerate}
    \itemsep0em 
    \item (Eigenvalue decomposition) For all $\Sigma\in\Sym(n)$, there exist $U\in\Orth(n)$ and $D\in\Diag(n)$ such that $\Sigma=UDU^\top$. By removing null eigenvalues, given $r=\rk(\Sigma)$, there also exist $U\in\St(n,r)$ and $D\in\Diag^*(r)$ such that $\Sigma=UDU^\top$.
    
    \item For all $X\in\R^{n\times k}$, $XX^\top$ is symmetric positive semi-definite and $\rk(X)=\rk(XX^\top)=\rk(X^\top X)$. In particular, if $X\in\R^{n\times k}_*$, then $X^\top X\in\GL(k)$.
    
    \item (Singular value decomposition) For all $M\in\R^{n\times k}$, denoting $r=\rk(M)\ls\min(n,k)$, there exist $U\in\Orth(n)$, $V\in\Orth(k)$ and $D=\begin{pmatrix}D_r & \mathbf{0}_{r,k-r}\\\mathbf{0}_{n-r,r} & \mathbf{0}_{n-r,k-r}\end{pmatrix}$ with $D_r\in\Diag^+(r)$ such that $M=UDV^\top$. The diagonal entries of $D$, that is the $D_{ii}'s$ for $i\in\{1,...,\min(n,k)\}$ are called the singular values of $M$.
    
    \item (Moore-Penrose inverse) For all $M\in\R^{n\times k}$, the unique matrix $M^-\in\R^{k\times n}$ satisfying $MM^-M=M$, $M^-MM^-=M^-$, $MM^-\in\Sym(n)$ and $M^-M\in\Sym(k)$ is called the pseudoinverse or Moore-Penrose inverse. In this work, we only use it for symmetric matrices $\Sigma\in\Sym(k)$. Given $r=\rk(\Sigma)$ and an eigenvalue decomposition $\Sigma=UDU^\top=P\Diag(D,0)P^\top$ with $U\in\St(k,r)$, $P=[U~U_\perp]\in\Orth(k)$, $D\in\Diag^*(r)$, it is easy to check that $\Sigma^-=UD^{-1}U^\top=P\Diag(D^{-1},0)P^\top$.
    Moreover, for all $V\in\St(n,k)$, $(V\Sigma V^\top)^-=V\Sigma^-V^\top$.
    
    \item (Symmetric square root) For all $\Sigma\in\Sym(n)$, if $\Sigma\gs0$, then there exists a unique matrix $A\in\Sym(n)$, $A\gs 0$ such that $A^2=\Sigma$. It is called the (symmetric) square root of $\Sigma$ and it is denoted $\sqrt{\Sigma}$ or $\Sigma^{1/2}$.
    
    \item (Non-symmetric square roots) For all $X,Y\in\R^{n\times k}$, $XX^\top=YY^\top$ if and only if there exists $U\in\Orth(k)$ such that $XU=Y$. Indeed, if $Y=XU$, then $YY^\top=XX^\top$. Conversely, if $X=(x_1,...,x_n)^\top$ and $Y=(y_1,...,y_n)^\top$ are such that $YY^\top=XX^\top$, then we can define the map $\varphi:\Span(x_1,...,x_n)\lto\Span(y_1,...,y_n)$ by $\varphi(x_i)=y_i$, extended by linearity. Indeed, it is well defined because if $\sum_{i=1}^n\lambda_ix_i=0$ or equivalently $X^\top\lambda=0$, then $\|Y^\top\lambda\|^2=\lambda^\top YY^\top\lambda=\lambda^\top XX^\top\lambda=0$ and $Y^\top\lambda=\sum_{i=1}^n{\lambda_iyi}=0$. Moreover, $\varphi$ is an isometric map, thus it is injective. Since $\rk(X)=\rk(XX^\top)=\rk(YY^\top)=\rk(Y)$, it is bijective so it is an isometry between two subvector spaces of $\R^k$. According to Witt's theorem, $\varphi$ extends to an isometry of $\R^k$. Thus the matrix $U=[\varphi]^\top$ in the canonical basis of $\R^k$ satisfies $XU=Y$.\\
\end{enumerate}

We manipulate several norms on matrices. For $M\in\R^{n\times k}$ and $x\in\R^k$, we denote:
\begin{enumerate}
    \itemsep0em
    \item (Euclidean norm) $\|M\|_2=\tr(M^\top M)^{1/2}$ and $\|x\|_2=(x^\top x)^{1/2}$.
    \item (Infinite norm) $\|M\|_\infty=\max_{1\ls i,j\ls n}|M_{ij}|$ and $\|x\|_\infty=\max_{1\ls i\ls n}|x_i|$.
    \item (Spectral norm or Schatten's infinite norm) $\|M\|_{\mathrm{S}}=\sup_{\substack{x\in\R^k\\\|x\|_2\ls 1}}\|Mx\|_2=\|\sigma(M)\|_\infty$, where $\sigma(M)$ is the vector of singular values of $M$.
\end{enumerate}
Without index, $\|\cdot\|$ generically denotes the norm on the tangent spaces associated to the Riemannian metric at hand.

\section{Preliminary concepts}\label{sec:bures-wasserstein_concepts}

In this section, we recall the basic definitions of geodesics and quotient spaces. We depart from works on Riemannian metrics on SPD matrices such as the affine-invariant, the log-Euclidean or the log-Cholesky metrics which are geodesically complete and where the injection domain of the exponential map is the entire manifold. We also need to introduce the notions of length and geodesic in a metric space \cite{Bridson99,Paulin14} and the notion of Riemannian orbit space \cite{Alekseevsky03}.

In Section \ref{subsec:bures-wasserstein_geodesics}, we introduce the concepts needed to study geodesics and minimizing geodesics. We recall that geodesics are defined in a metric space, in a manifold endowed with an affine connection and in a Riemannian manifold, where the two previous notions coincide. Then we introduce our notations for preimages of the exponential map, logarithms and related notions. We define the cut time and the injection domain, extending the definition usually given in complete metric spaces, and we explain where to be cautious. In Section \ref{subsec:bures-wasserstein_quotient_spaces}, we recall the notions of quotient metric space, quotient Riemannian manifold and Riemannian orbit space. We introduce the vocabulary, notations and results that we use in the next sections.

\subsection{Geodesics}\label{subsec:bures-wasserstein_geodesics}

\begin{definition}[Curve]
Let $\mc{M}$ be a topological space. A curve on $\mc{M}$ is a continuous map $c:I\lto\mc{M}$, where $I$ is an interval of $\R$. When $I$ is a segment of $\R$, we may call $c$ a segment.
\end{definition}

\subsubsection{Geodesics in a metric space}\label{subsubsec:bures-wasserstein_geodesics_metric_space}

\begin{definition}[Length, length distance, length space]\cite{Paulin14}
Let $(\mc{M},d)$ be a metric space.
\begin{enumerate}
    \itemsep0em
    \item (Length) Let $c:[a,b]\subseteq I\lto\mc{M}$ be a curve. The length of $c$ is defined by $L(c)=\sup\sum_{k=0}^pd(c(t_k),c(t_{k+1}))\in[0,+\infty]$ over all subdivisions $a=t_0\ls t_1\ls...\ls t_p\ls t_{p+1}=b$. We say that $c:I\lto\mc{M}$ is rectifiable when for all $a\ls b$ in $I$, $L(c_{|[a,b]})$ is finite.
    \item (Length distance) The length distance between $x\in\mc{M}$ and $y\in\mc{M}$ is defined by $d_L(x,y)=\inf L(c)\gs d(x,y)$ over all rectifiable curves $c:[0,1]\lto\mc{M}$ from $x=c(0)$ to $y=c(1)$. If $\mc{M}$ is connected by rectifiable curves, then $d_L$ is a distance.
    \item (Length space) We say that $(\mc{M},d)$ is a length space when $d=d_L$.
\end{enumerate}
\end{definition}

\begin{lemma}[Length is additive and continuous]\cite[Proposition 1.20]{Bridson99}\label{lemma:length}
Let $c:[a,b]\lto\mc{M}$ be a rectifiable curve.
\begin{enumerate}
    \itemsep0em
    \item For all $t\in[a,b]$, $L(c)=L(c_{|[a,t]})+L(c_{|[t,b]})$.
    \item The map $f:t\in[a,b]\lmto L(c_{|[a,t]})$ is non-decreasing and continuous.
\end{enumerate}
\end{lemma}

A classical example is the sphere $\Sph^2\subset\R^3$ endowed with the Euclidean distance $d$. The distances between the north pole $N=(0,0,1)$ and the south pole $S=(0,0,-1)$ are $d(N,S)=2$ and $d_L(N,S)=\pi$.

\begin{definition}[Geodesics in a metric space]\cite{Paulin14}
Let $(\mc{M},d)$ be a metric space.
\begin{enumerate}
    \itemsep0em
    \item (Constant speed) We say that $c:I\lto\mc{M}$ is a curve parameterized at constant speed when there exists $v\gs 0$ such that for all $t\ls t'$ in $I$, $L(c_{|[t,t']})=v(t'-t)$.
    \item (Unit speed) We say that $c:I\lto\mc{M}$ is a curve parameterized at unit speed (or by arc length) when for all $t\ls t'$ in $I$, $L(c_{|[t,t']})=t'-t$.
    \item (Globally minimizing) We say that $c:I\lto\mc{M}$ is (globally) minimizing when for all $a\ls b$ in $I$, $L(c_{|[a,b]})=d(c(a),c(b))$.
    \item (Locally minimizing) We say that $c:I\lto\mc{M}$ is locally minimizing when for all $t\in I$, there exists a neighborhood $I_0\subseteq I$ of $t$ such that $c_{|I_0}$ is minimizing.
    \item (Geodesic) A geodesic is a locally minimizing curve of constant speed. A minimizing geodesic is a globally minimizing curve of constant speed.
    \item (Geodesic space) A geodesic (metric) space is a length space such that there exists a minimizing geodesic between any two points.
\end{enumerate}
\end{definition}

\subsubsection{Geodesics of an affine connection}\label{subsubsec:bures-wasserstein_geodesics_affine_connection}

In the following notions, the definition interval and the definition domain are of interest when the manifold is not geodesically complete and the set of preimages is of interest when the exponential map is not injective.

\begin{definition}[Geodesics of an affine connection]\cite{Paulin14}
Let $(\mc{M},\nabla)$ be a smooth manifold equipped with an affine connection $\nabla:\Gamma(T\mc{M})\times\Gamma(T\mc{M})\lto\Gamma(T\mc{M})$.
\begin{enumerate}
    \item (Geodesic) A geodesic (or self-parallel curve) is a solution $\gamma:I\lto\mc{M}$ of the second-order equation $\forall t\in I,\nabla_{\dot\gamma(t)}\dot\gamma=0$. The maximal solution satisfying the initial condition $\dot\gamma(0)=v$ for $v\in T_x\mc{M}$ is denoted $\gamma_{(x,v)}:I_{x,v}\lto\mc{M}$.
    
    \item (Definition interval) We call $I_{x,v}\subseteq\R$ the definition interval of $\gamma_{(x,v)}$. It is the maximal interval of $\R$ on which the geodesic $\gamma_{(x,v)}$ is defined.
    
    In a geodesically complete manifold, $I_{x,v}=\R$.
    
    \item (Exponential map) The exponential map is defined by $\Exp:v\lmto\gamma_{(x,v)}(1)$ on the open set $\bigcup_{x\in\mc{M}}\{v\in T_x\mc{M}|\,1\in I_{x,v}\}\subseteq T\mc{M}$. The exponential map at $x$, defined by $\Exp_x(v)=\Exp(v)$ for $v\in T_x\mc{M}$ such that $1\in I_{x,v}$, is a diffeomorphism from a neighborhood of $0$ in $T_x\mc{M}$ to a neighborhood of $x$ in $\mc{M}$.
    
    \item (Definition domain) The definition domain of the exponential map if $\mc{D}_x=\{v\in T_x\mc{M}|\,1\in I_{x,v}\}$.
    
    \item (Preimage) A preimage of $y\in\mc{M}$ by $\Exp_x$ is a vector $v\in T_x\mc{M}$ such that $1\in I_{x,v}$ and $\Exp_x(v)=y$. The set of preimages of $y$ from $x$ is denoted $\mc{P}re_x(y)=\Exp_x^{-1}(\{y\})=\{v\in T_x\mc{M}|\,1\in I_{x,v}\mathrm{~and~}\Exp_x(v)=y\}$.
    
    Note that there might be none, one, several or infinitely many elements in $\mc{P}re_x(y)$.
    
    \item (Geodesic from $x$ to $y$) A geodesic from $x$ to $y$ is a geodesic $\gamma_{(x,v)}$ such that $v\in\mc{P}re_x(y)$. They are bijectively indexed by $\mc{P}re_x(y)$ so we denote them $\gamma_{x\to y}^v=\gamma_{(x,v)}$ for $v\in\mc{P}re_x(y)$.
    
    When there exists a unique preimage of $y$ from $x$, we simply denote the geodesic $\gamma_{x\to y}$.
\end{enumerate}
\end{definition}

\subsubsection{Geodesics of a Riemannian metric}\label{subsubsec:bures-wasserstein_geodesics_riemannian}

Any Riemannian manifold is equipped with a natural affine connection called the Levi-Civita connection. Moreover, we recall in the following definition that a connected Riemannian manifold is a metric space and even a length space.

\begin{definition}[Length, Riemannian distance]\cite{Paulin14}
Let $(\mc{M},g)$ be a Riemannian manifold. For $v\in T\mc{M}$, we denote its norm $\|v\|=\sqrt{g(v,v)}$.
\begin{enumerate}
    \item (Length) The length of a $\mc{C}^1$ curve $c:[a,b]\lto\mc{M}$ is defined by $L(c)=\int_0^1{\|\dot{c}(t)\|dt}$. This definition extends to piecewise $\mc{C}^1$ curves. The length is independent from the parametrization of the curve.
    \item (Distance) The Riemannian distance between $x,y\in\mc{M}$ is defined by $d(x,y)=\inf L(c)$ over all piecewise $\mc{C}^1$ curves $c:[0,1]\lto\mc{M}$ from $x=c(0)$ to $y=c(1)$. If $\mc{M}$ is connected, then the Riemannian distance is a distance (defining the topology of $\mc{M}$), the Riemannian length and the metric length coincide on piecewise $\mc{C}^1$ curves, and $(\mc{M},d)$ is a length space.
\end{enumerate}
\end{definition}

Fortunately, the two notions of geodesics coincide: a piecewise $\mc{C}^1$ curve is a geodesic for the Levi-Civita connection if and only if it is a geodesic for the Riemannian distance. Then, it is natural to ask what are the \textit{globally} minimizing geodesics. A few concepts can be introduced to formalize this question: injectivity radius, cut time, cut locus, injection domain. Moreover, among the preimages $v\in\mc{P}re_x(y)$ of the exponential map of the Levi-Civita connection, those which satisfy $\|v\|=d(x,y)$ are called Riemannian logarithms. When there exists a unique Riemannian logarithm, the logarithm map can be defined. 

\begin{definition}[Geodesics of a Riemannian metric]
Let $(\mc{M},g)$ be a connected Riemannian manifold. Let $d$ be the Riemannian distance and $\nabla$ be the Levi-Civita connection. We denote $\mc{B}_x(\varepsilon)=\{v\in T_x\mc{M}|\|v\|<\varepsilon\}\subset T_x\mc{M}$ the centered open ball of $T_x\mc{M}$ of radius $\varepsilon>0$.
\begin{enumerate}
    \item (Geodesics) \cite[Proposition 3.14]{Paulin14} A curve $\gamma:I\lto\mc{M}$ is a geodesic of $(\mc{M},d^g)$ (locally length-minimizing curve of constant speed) if and only if it is a geodesic of $(\mc{M},\nabla^g)$ (self-parallel curve). Such a curve is called a geodesic of $(\mc{M},g)$.
    
    \item (Injectivity radius) \cite{Paulin14} The injectivity radius at $x\in\mc{M}$ is defined by $\mathrm{inj}(x)=\sup\varepsilon$ over all $\varepsilon>0$ such that $\exp_x$ is a diffeomorphism from $\mc{B}_x(\varepsilon)\subset T_x\mc{M}$ to its image. The injectivity ball is $\mc{B}_x(\mathrm{inj}(x))\subset T_x\mc{M}$.
    
    For $v\in T_x\mc{M}$ of norm 1 and $t_0\in[0,\mathrm{inj}(x))$, the map $\gamma_{(x,v)}:t\in[0,t_0]\lmto\exp_x(tv)\in\mc{M}$ is the unique geodesic between $x$ and $\gamma_{(x,v)}(t_0)$. It is globally minimizing.
    
    The injectivity radius of $\mc{M}$ is defined by $\mathrm{inj}(\mc{M})=\inf_{x\in\mc{M}}\mathrm{inj}(x)$.
    
    \item (Cut time) The cut time at $x\in\mc{M}$ in the direction $v\in T_x\mc{M}$, $\|v\|=1$, is defined by $t_{cut}(x,v)=\sup\{t\in I_{x,v}|d(x,\Exp_x(tv))=t\}\in(0,+\infty]$. Note that the geodesic $\gamma_{(x,v)}$ need not be minimizing on $(-t_{cut}(x,-v),t_{cut}(x,v))$.
    
    \item (Tangential cut locus) The tangential cut locus at $x$ is the set $TCL(x)=\{t_{cut}(x,v)v|\,v\in T_x\mc{M},\|v\|=1,t_{cut}(x,v)<+\infty\}$.
    
    \item (Injection domain) The injection domain of the exponential map at $x$ is the set $\mathrm{Inj}(x)=\{tv|t\in[0,t_{cut}(x,v)),v\in T_x\mc{M},\|v\|=1\}\subseteq\mc{D}_x$. The injection ball is included in the injection domain.
    
    \item (Logarithms) We call (Riemannian) logarithm of $y\in\mc{M}$ from $x\in\mc{M}$ a preimage $v\in\mc{P}re_x(y)\subseteq T_x\mc{M}$ of $y$ from $x$ by the exponential map such that $\|v\|=d(x,y)$. We denote $\mc{L}og_x(y)\subseteq\mc{P}re_x(y)$ the set of logarithms of $y$ from $x$. In particular, the geodesic $\gamma_{x\to y}^v$ joining $x$ to $y$ with initial speed $v\in\mc{L}og_x$ is minimizing on $[0,1]$.
    
    \item (Logarithm map) Denoting $\mc{U}_x\subseteq\mc{M}$ the subset of points $y$ such that there exists a unique logarithm of $y$ from $x$, this defines a map $\Log_x:\mc{U}_x\lto T_x\mc{M}$. In particular, $\mc{B}_x(\mathrm{inj}(x))\subseteq\mc{U}_x$ and $\Log_x:\Exp_x(\mc{B}_x(\varepsilon))\lto \mc{B}_x(\varepsilon)$ is a diffeomorphism for all $\varepsilon<\mathrm{inj}(x)$.
\end{enumerate}
\end{definition}

\begin{remark}
The cut time is usually defined in complete manifolds only \cite{DoCarmo92}. Although the definition still holds, some basic results may fail in non-complete manifolds. For example, if $I_{x,v}\ne\R$ and $t_{cut}(x,v)<+\infty$, then $t_{cut}(x,v)$ need not belong to $I_{x,v}$ and $\Exp_x(t_{cut}(x,v)v)$ need not be defined. Then the cut locus, which is the image of the tangential cut locus by the exponential map in complete manifolds, should be defined differently. One possible definition could simply forget the vectors $v\in T_x\mc{M}$ such that $t_{cut}(x,v)\notin I_{x,v}$ as well as the definition in complete manifolds forgets about the vectors $v\in T_x\mc{M}$ such that $t_{cut}(x,v)=+\infty$. A maybe more satisfying definition could rely on the metric completion of the space: the cut point in direction $v$ could be the limit of $\Exp_x(tv)$ when $t$ tends to $t_{cut}(x,v)$ if this limit exists.
\end{remark}

Studying the geodesics of a space means (at least) determining precisely an expression of the geodesic $\gamma_{(x,v)}$, the maximal domain $I_{x,v}$, the injectivity radiuses $\mathrm{inj}(x)$ and $\mathrm{inj}(\mc{M})$, the cut time $t_{cut}(x,v)$, the preimages of the exponential map $\mc{P}re_x(y)$, an expression of the geodesics $\gamma_{x\to y}^v$, the Riemannian logarithms $\mc{L}og_x(y)$, the logarithm map $\Log_x$ and its definition domain $\mc{U}_x$.

The goal of Sections \ref{sec:bures-wasserstein_sym^+(n)} and \ref{sec:bures-wasserstein_sym^+(n,k)} is to clarify and complete the knowledge on geodesics of the Bures-Wasserstein metric on the manifolds $\Sym^+(n)$ and $\Sym^+(n,k)$.

\subsection{Quotient spaces}\label{subsec:bures-wasserstein_quotient_spaces}
\subsubsection{Quotient distance}\label{subsubsec:bures-wasserstein_quotient_distance}

\begin{definition}[Quotient distance]
Let $(\mc{M},d)$ be a metric space. Let $G$ be a group acting isometrically on $(\mc{M},d)$. We denote $\mc{M}^0=\mc{M}/G$ and we define the map $d^0:(Gx,Gy)\in\mc{M}^0\times\mc{M}^0\lmto d(Gx,Gy)=\inf_{g\in G}d(gx,y)\in[0,+\infty)$. If the orbits are closed, then $d_0$ is a distance on $\mc{M}^0$ called the quotient distance.
\end{definition}

\begin{remark}
There exists a more general notion of quotient pseudo-distance when the action is not isometric. In this work, the actions are isometric so we don't detail it.
\end{remark}

\begin{remark}
The hypotheses of the definition are satisfied if the following conditions hold together:
\begin{enumerate}[label=$\cdot$]
    \itemsep0em
    \item $G$ is a locally compact topological group (e.g. a Lie group),
    \item $(\mc{M},d)$ is a metric space (e.g. a Riemannian manifold),
    \item the action of $G$ on $(\mc{M},d)$ is continuous and isometric,
    \item the action is proper, i.e. (with the previous assumptions \cite[III.4.4 Proposition 7]{Bourbaki71}) for all $x,y\in\mc{M}$, there exist respective neighborhoods $\mc{V}_x,\mc{V}_y$ such that the set $K=\{g\in G|g\mc{V}_x\cap \mc{V}_y\ne\emptyset\}$ is relatively compact.
\end{enumerate}
Indeed, let $(g_nx)_{n\in\N}$ be a sequence in $Gx$ tending to $y\in\mc{M}$. Then there exists $n_0\in\N$ such that for all $n\gs n_0$, $g_nx\in\mc{V}_y$ so $g_n\mc{V}_x\cap\mc{V}_y\ne\emptyset$, i.e. $g_n\in K$. Since $\mc{M}$ is a metric space and $K$ is relatively compact, there exists a subsequence $(g_{\varphi(n)})$ converging to $g\in G$. Since the action is continuous, $g_{\varphi(n)}x$ converges to $gx$. Therefore, $y=gx\in Gx$ and $Gx$ is sequentially closed, hence closed.
\end{remark}

\begin{definition}[Registered points]
We say that $x,y\in\mc{M}$ are registered points when $d(x,y)=d^0(Gx,Gy)$.
\end{definition}

\begin{lemma}[Length in a quotient metric space]\label{lemma:length_quotient}
We denote $L$ the length on both $\mc{M}$ and $\mc{M}^0=\mc{M}/G$, 
and $\pi:\mc{M}\lto\mc{M}^0$ the canonical projection. For all curve $c:[0,1]\lto\mc{M}$, we have $L(\pi\circ c)\ls L(c)$. In particular, if $x=c(0)$ and $y=c(1)$ are registered and if $c$ is minimizing, then $\pi\circ c$ is minimizing and $L(\pi\circ c)=L(c)=d(x,y)=d^0(\pi(x),\pi(y))$.
\end{lemma}

\begin{proof}
By definition $L(\pi\circ c)\ls\sup\sum_{k=0}^p\underset{\ls d(c(t_k),c(t_{k+1}))}{\underbrace{d^0(\pi(c(t_k)),\pi(c(t_{k+1})))}}\ls L(c)$. If $x=c(0)$ and $y=c(1)$ are registered, then $d(x,y)=d^0(\pi(x),\pi(y))\ls L(\pi\circ c)$. If $c$ is minimizing, then $d(x,y)=L(c)\gs L(\pi\circ c)$. Thus if both hold, then $L(\pi\circ c)=L(c)=d(x,y)=d^0(\pi(x),\pi(y))$ and $\pi\circ c$ is minimizing.
\end{proof}

\subsubsection{Quotient Riemannian metric}\label{subsubsec:bures-wasserstein_quotient_riemannian_metric}

\begin{definition}[Quotient metric]
Let $(\mc{M},g)$ be a Riemannian manifold. Let $G$ be a Lie group acting smoothly, properly, freely and isometrically on $(\mc{M},g)$. Then there exists a unique smooth structure on $\mc{M}^0=\mc{M}/G$ such that the quotient map $\pi:\mc{M}\lto\mc{M}^0$ is a submersion \cite{Lee12}. Thus, one can define for all $x\in\mc{M}$:
\begin{enumerate}
    \itemsep0em
    \item (Vertical space) $\mc{V}_x=T_x\mc{M}_x=\ker d_x\pi$ where $\mc{M}_x=\pi^{-1}(x)$ is a submanifold of $\mc{M}$,
    \item (Horizontal space) $\mc{H}_x=\mc{V}_x^\perp$ so that $T_x\mc{M}=\mc{V}_x\oplus\mc{H}_x$,
    \item (Horizontal lift) $\#_x:T_{\pi(x)}\mc{M}^0\lto\mc{H}_x$ the inverse isomorphism of $(d_x\pi)_{|\mc{H}_x}:\mc{H}_x\lto T_{\pi(x)}\mc{M}^0$,
    \item (Quotient metric) $g^0_{\pi(x)}:(v,w)\in T_{\pi(x)}\mc{M}^0\times T_{\pi(x)}\mc{M}^0\lmto g_x(v^\#_x,w^\#_x)\in\R$.
\end{enumerate}
This is a particular case of a Riemannian submersion \cite{ONeill66}. The Riemannian distance of the quotient metric $g^0$ is the quotient distance of $d$, defined by $d^0(\pi(x),\pi(y))=d(Gx,Gy)=\inf_{g\in G}d(gx,y)$.
\end{definition}

The fundamental theorem on geodesics of a quotient metric is the following.

\begin{theorem}[Geodesics of a quotient metric]\label{thm:geodesics_quotient}\cite{ONeill66}
The projection of a horizontal geodesic is a horizontal geodesic and their lengths coincide on any segment. More precisely, let $x\in\mc{M}$ and $v\in T_{\pi(x)}\mc{M}^0$. Then $I_{x,v^\#_x}\subseteq I_{\pi(x),v}$ and for all $t\in I_{x,v^\#_x}$, $\Exp_{\pi(x)}(tv)=\pi(\Exp_x(tv^\#_x))$.
\end{theorem}

\subsubsection{Riemannian orbit spaces}\label{subsubsec:bures-wasserstein_orbit_spaces}

When the action of $G$ on $\mc{M}$ is not free, the orbit space $\mc{M}/G$ is not a Riemannian manifold in general. This was studied in \cite{Alekseevsky03}. We briefly recall the main facts that we use in this work.

\begin{theorem}[Riemannian geometry of orbit spaces]\cite{Alekseevsky03}\label{thm:riem_orbit_space}
Let $(\mc{M},g)$ be a connected complete Riemannian manifold. Let $G$ be a Lie group acting smoothly, properly and isometrically on $(\mc{M},g)$. We denote $\pi:\mc{M}\lto\mc{M}/G$ the canonical surjection. For a Lie subgroup $H$ of $G$, we denote $(H)=\{gHg^{-1}|g\in G\}$ the conjugacy class of $H$ and $\mc{M}_{(H)}$ the set of points $x\in\mc{M}$ such that the stabilizer of $x$, $\Stab(x)=\{g\in G|gx=x\}$, belongs to $(H)$.
\begin{enumerate}
    \itemsep0em
    
    \item $\mc{M}_{(H)}$ is a smooth submanifold of $\mc{M}$.
    
    \item We denote $(\mc{M}/G)_{(H)}=\pi(\mc{M}_{(H)})=\mc{M}_{(H)}/G$ the (isotropy) stratum of type $(H)$. Then $\pi_{(H)}:=\pi_{|\mc{M}_{(H)}}:\mc{M}_{(H)}\lto(\mc{M}/G)_{(H)}$ is a smooth fiber bundle with fiber type $G/H$.
    
    \item The isotropy strata form a partition of $\mc{M}/G$.
    
    \item $(\mc{M}/G,d^0)$ is a complete metric space and a length space.
\end{enumerate}
\end{theorem}

The definition of the vertical space still holds while the horizontal space is replaced by the normal space \cite[VI.29.2]{Michor08}.

\begin{definition}[Vertical space, normal space]
We take the notations of Theorem \ref{thm:riem_orbit_space}. Let $x\in\mc{M}_{(H)}$. Thus $\mc{M}_x=\pi^{-1}(\pi(x))$ is a submanifold of $\mc{M}_{(H)}$.
\begin{enumerate}
    \itemsep0em
    \item (Vertical space) $\mc{V}_x=T_x\mc{M}_x\subseteq T_x\mc{M}_{(H)}\subseteq T_x\mc{M}$.
    \item (Normal space) $\mc{N}_x=\mc{V}_x^\perp\subseteq T_x\mc{M}$ so that $T_x\mc{M}=\mc{V}_x\oplus\mc{N}_x$.
\end{enumerate}
\end{definition}

Note that $\mc{V}_x$ and $\mc{N}_x$ need not have a constant dimension. We recall a result that we use later.

\begin{lemma}[Geodesics in a Riemannian orbit space]\cite[Lemma 3.5]{Alekseevsky03}\label{lemma:minimizing_geodesics_strata}
We take the notation of Theorem \ref{thm:riem_orbit_space}. Given Lie subgroups $H_1,H_2$ of $G$, we denote $(H_1)\ls(H_2)$ if $H_1$ is conjugated so a subgroup of $(H_2)$. Let $\gamma:[0,1]\lto\mc{M}/G$ be a minimizing curve. For $t\in[0,1]$, let $(\mc{M}/G)_{(H_t)}$ denote the stratum of $\gamma(t)$. Then, for all $t\in(0,1)$, $(H_t)\ls(H_0)$ and $(H_t)\ls(H_1)$.
\end{lemma}

We are now well prepared to study the construction of the orbit space of covariance matrices (Section \ref{sec:bures-wasserstein_geometry}), the geodesics within each stratum (Sections \ref{sec:bures-wasserstein_sym^+(n)} and \ref{sec:bures-wasserstein_sym^+(n,k)}) and the minimizing geodesics in the whole space (Section \ref{sec:bures-wasserstein_cov(n)}).

\section{Bures-Wasserstein geometry of covariance matrices}\label{sec:bures-wasserstein_geometry}

We denote the set of symmetric positive semi-definite matrices or covariance matrices by $\Cov(n)=\{\Sigma\in\Sym(n)|\Sigma\gs 0\}=\{XX^\top|X\in\Mat(n)\}\subset\Mat(n)$. It is a complete metric subspace of the vector space of $n\times n$ square matrices equipped with the Euclidean distance $d^{\mathrm{E}}(\Sigma,\Lambda)=\|\Sigma-\Lambda\|_2=\tr((\Sigma-\Lambda)^2)^{1/2}$.

In this section, we recall that this set can also be described as the orbit space of the manifold $\Mat(n)$ of $n\times n$ matrices quotiented on the right by the orthogonal group $\Orth(n)$. We recall that the quotient topology coincides with the Euclidean topology and we recall the expression of the quotient distance known as the Bures-Wasserstein distance. We insist on the definition of the strata, which are investigated in the following sections.

\subsection{The quotient geometry of covariance matrices}\label{subsec:bures-wasserstein_quotient_geometry}

The \textbf{group action} of the orthogonal group $\Orth(n)$ on the vector space of square matrices $\Mat(n)$ is $(X,U)\in\Mat(n)\times\Orth(n)\lmto XU\in\Mat(n)$. It is smooth, proper and isometric for the Euclidean distance.

The \textbf{stabilizer} (or isotropy group) of a matrix $X\in\Mat(n)$ is $\Stab(X)=\{U\in\Orth(n)|XU=X\}$. If we denote $k=\rk(X)$, it is well known that $X$ is equivalent to the matrix $J_k=\begin{pmatrix}I_k & 0\\0 & 0\end{pmatrix}$, i.e. there exist $P,Q\in\GL(n)$ such that $X=PJ_kQ$. Then it is clear that $H_k:=\Stab(J_k)=\{\begin{pmatrix}I_k & 0\\0 & U\end{pmatrix}|U\in\Orth(n-k)\}$ and $\Stab(X)=Q^{-1}\,\Stab(J_k)\,Q$ with $\dim\Stab(X)=\dim\Orth(n-k)=\frac{(n-k)(n-k-1)}{2}$. Hence two matrices have conjugate stabilizers if and only if they have the same rank. Note that $(H_k)\ls(H_l)$ if and only if $k\gs l$.

The \textbf{orbit} of $X$ is $\Orb(X)=\{XU|U\in\Orth(n)\}\simeq\Orth(n)/\Stab(X)$. Its dimension is $\dim\Orb(X)=\dim\Orth(n)-\dim\Stab(X)=nk-\frac{k(k+1)}{2}$. Given a result recalled in the introduction, we have $\Orb(X)=\{Y\in\Mat(n)|YY^\top=XX^\top\}$.

The \textbf{orbit space} $\Mat(n)/\Orth(n)=\{\Orb(X)|X\in\Mat(n)\}$ is thus in bijection with the set of covariance matrices $\Cov(n)=\{XX^\top|X\in\Mat(n)\}$ by the map $\Orb(X)\lmto XX^\top$. 

The \textbf{orbit strata} of $\Mat(n)/\Orth(n)$ are the sets of points that have conjugate stabilizers \cite{Alekseevsky03}, i.e. that have the same rank here. This give a manifold structure to $\Mat(n)_{(H_k)}=\R^{n\times n}_k$. The strata of $\Mat(n)/\Orth(n)$ are $\R^{n\times n}_k/\Orth(n)$, or equivalently the strata of covariance matrices are the sets of symmetric positive semi-definite matrices of fixed rank $\Sym^+(n,k)=\Cov(n)\cap\R^{n\times n}_k$. The principal/regular stratum is the set of Symmetric Positive Definite (SPD) matrices $\Sym^+(n)=\Cov(n)\cap\GL(n)$. Finally, $\pi_k:=\pi_{(H_k)}:\R^{n\times n}_k\lto\Sym^+(n,k)$ is a smooth fiber bundle with fiber type $\St(n,k)=\Orth(n)/\Orth(n-k)$, in particular it is a submersion.

\subsection{The Bures-Wasserstein distance}\label{subsec:bures-wasserstein_distance}

Since the group action is continuous, proper and isometric, the Euclidean distance descends to a distance on $\Mat(n)/\Orth(n)$. Via the bijection $\Orb(X)\in\Mat(n)/\Orth(n)\lmto XX^\top\in\Cov(n)$, it is usually expressed as a distance on covariance matrices. It is known as the Bures-Wasserstein distance \cite{Dowson82,Olkin82,Bhatia19}.

\begin{definition}[Bures-Wasserstein distance]\label{def:BW_distance}\cite{Bhatia19}
The Bures-Wasserstein distance between $\Sigma$ and $\Lambda$ is defined by:
\begin{align}
    d^{\mathrm{BW}}(\Sigma,\Lambda)&=\inf_{\substack{X,Y\in\Mat(n)\\XX^\top=\Sigma,YY^\top=\Lambda}}d^{\mathrm{E}}(X,Y)=\inf_{R\in\Orth(n)}d^{\mathrm{E}}(\Sigma^{1/2},\Lambda^{1/2}R)\\
    &=\tr(\Sigma+\Lambda-2(\Sigma^{1/2}\Lambda\Sigma^{1/2})^{1/2})^{1/2}.
\end{align}
If $XX^\top=\Sigma$ and $YY^\top=\Lambda$, let $R\in\Orth(n)$ such that $X^\top Y=(X^\top\Lambda X)^{1/2}R$. Then $d^{\mathrm{E}}(X,YR^\top)^2=\|YR^\top-X\|_2^2=\tr(XX^\top+YY^\top-2X^\top YR^\top)=d^{\mathrm{BW}}(\Sigma,\Lambda)^2$.
\end{definition}

As the quotient of a length space, the space of covariance matrices endowed with the Bures-Wasserstein metric is a length space. It is even a complete geodesic metric space \cite[Proposition 3.1.(1)]{Alekseevsky03}.

The following result seems elementary although we did not find a clear reference in the literature.

\begin{lemma}[\chixlemmaeucbwtoptitle]\label{lemma:bures-wasserstein_topologies}
\chixlemmaeucbwtop\\
\begin{NDLR}
See the proof of Lemma \ref{lemma:bures-wasserstein_topologies} in Appendix \ref{subsec:proof:lemma:bures-wasserstein_topologies}.
\end{NDLR}
\end{lemma}

\subsection{Topology, metric and smooth structure of the strata}\label{subsec:bures-wasserstein_topology}

The set of symmetric positive definite (SPD) matrices $\Sym^+(n)=\{\Sigma\in\Sym(n)|\,\sp(\Sigma)\subset(0,+\infty)\}$ is an open set of the vector space of symmetric matrices, hence it has a natural structure of smooth manifold. This topology clearly coincides with the topology induced by $(\Cov(n),d^{\mathrm{E}})$, thus it also coincides with the topology induced by $(\Cov(n),d^{\mathrm{BW}})$.

The set $\Sym^+(n,k)$ is in bijection with $\R^{n\times n}_k/\Orth(n)$. The Euclidean distance on $\R^{n\times n}_k$ clearly descends to the Bures-Wasserstein distance on $\Sym^+(n,k)$. Therefore, the quotient topology coincides with the Bures-Wasserstein topology induced by $(\Cov(n),d^{\mathrm{BW}})$, thus also with the Euclidean topology induced by $(\Cov(n),d^{\mathrm{E}})$ and $(\Sym(n),d^{\mathrm{E}})$.

However, the Riemannian geometry is difficult to study via the submersion $\pi_k:\R^{n\times n}_k\lto\Sym^+(n,k)$. Indeed, it is the projection of a bundle of fiber $\St(n,k)\simeq\Orth(n)/\Orth(n-k)$ which is not a Lie group. Fortunately, the set $\Sym^+(n,k)$ is also in bijection with the quotient manifold $\R^{n\times k}_*/\Orth(k)$ \cite[Proposition 2.1]{Massart20}, where $\R^{n\times k}_*$ is the open set of matrices of full rank in $\R^{n\times k}$. We recall this quotient geometry in Table \ref{tab:manifold_Sym^+(n,k)}. The quotient distance induced on $\Sym^+(n,k)$ is the Bures-Wasserstein distance \cite[Proposition 5.1]{Massart20}. In particular, the quotient topology coincides with the previous ones. This bijection naturally provides a smooth structure on $\Sym^+(n,k)$. Above all, the submersion $\Tilde{\pi}_k:\R^{n\times r}_*\lto\Sym^+(n,k)$ is the projection of a principal fiber bundle. Hence, it is much more convenient to study the Bures-Wasserstein Riemannian geometry of $\Sym^+(n,k)$ via the Riemannian submersion $\Tilde{\pi}_k$. This is exactly what is done in \cite{Massart20}.

To summarize, the strata $\Sym^+(n,k)$ are smooth connected manifolds and the regular stratum $\Sym^+(n)$ is a dense open set in $\Cov(n)$.

    \begin{table}[h]
    \centering
    \begin{tabular}{|c|c|}
        \hline
        Set & $\Sym^+(n,k)=\Cov(n)\cap\R^{n\times n}_k$ \\
        \hline
        Smooth manifold & $\R^{n\times k}_*/\Orth(k)$ \\
        \hline
        Group action & $\fun{\R^{n\times k}_*\times\Orth(k)}{\R^{n\times k}_*}{(X,U)}{XU}$ \\
        \hline
        Orbit & $\Orb(X)=\{Y\in\R^{n\times k}_*|YY^\top=XX^\top\}$ \\
        \hline
        Identification & $\fun{\R^{n\times k}_*/\Orth(k)}{\Sym^+(n,k)}{\Orb(X)}{XX^\top}$ \\
        \hline
        Submersion & $\pi_{\Sym^+(n,k)}:\fun{\R^{n\times k}_*}{\Sym^+(n,k)}{X}{XX^\top}$ \\
        \hline
    \end{tabular}
    \caption{Smooth manifold structure of $\Sym^+(n,k)$.}
    \label{tab:manifold_Sym^+(n,k)}
    \end{table}

The Riemannian geometry of the principal stratum was extensively studied \cite{Takatsu10,Takatsu11,Malago18,Bhatia19,Oostrum20,Thanwerdas22-LAA} and the Riemannian geometry of the other strata was recently well detailed \cite{Massart19,Massart20}. However, there remain missing formulae and open questions about the geodesics in each stratum: injection domain, preimages, and explicit formulae of the horizontal lift, the exponential map and logarithms in the base space $\Sym^+(n,k)$ of the principal fiber bundle $\R^{n\times k}_*\lto\Sym^+(n,k)$. We precisely answer these questions in Section \ref{sec:bures-wasserstein_sym^+(n)} (full-rank matrices) and Section \ref{sec:bures-wasserstein_sym^+(n,k)} (low-rank matrices). Furthermore, we contribute in Section \ref{sec:bures-wasserstein_cov(n)} the minimizing geodesics for the Bures-Wasserstein \textit{distance} between different strata and the condition of uniqueness of the geodesic between two points.

\section{Geodesics of the Bures-Wasserstein metric on $\Sym^+(n)$}\label{sec:bures-wasserstein_sym^+(n)}

In this section, we give complements and new results on the Bures-Wasserstein geodesics on SPD matrices $\Sym^+(n)$. The quotient structure is well known, as well as the exponential map \cite{Malago18} and the injectivity radius \cite{Massart20}. The definition interval of the geodesic was implicitly described in \cite{Malago18} as the connected component of 0 in a subset or $\R$ so we give it explicitly here. It was proved in \cite{Bhatia19} that there exists a preimage which is a logarithm. We prove the uniqueness of the preimage and the logarithm based on a result of \cite{Massart20} on $\Sym^+(n,k)$ applied for $k=n$. Moreover, we contribute the cut time, thus the injection domain. After Theorem \ref{thm:bw_geodesics_sym^+(n)} on Bures-Wasserstein geodesics, we show on an example that we already know some geodesics between degenerate matrices that cross the main stratum of SPD matrices.

\begin{definition}[Bures-Wasserstein metric on $\Sym^+(n)$]\cite{Malago18,Bhatia19,Oostrum20}
The Bures-Wasserstein metric on $\Sym^+(n)$ is the quotient Riemannian metric induced by the submersion $\pi:X\in\GL(n)\lmto XX^\top\in\Sym^+(n)$ and the Frobenius metric on $\GL(n)$. Let $X\in\GL(n)$ such that $XX^\top=\Sigma$ and let $V\in T_\Sigma\Sym^+(n)\equiv\Sym(n)$. The quotient operations are:
\begin{enumerate}
    \itemsep0em
    \item (Vertical space) $\mc{V}_X=\ker d_X\pi=X\,\Skew(n)$,
    \item (Horizontal space) $\mc{H}_X=\Sym(n)\,X$,
    \item (Horizontal lift) $V^\#_X=\mc{S}_\Sigma(V)X\in\mc{H}_X$,
    \item (Bures-Wasserstein metric) $g^{BW(n)}_\Sigma(V,V)=\tr(\mc{S}_\Sigma(V)\Sigma\mc{S}_\Sigma(V))$,
\end{enumerate}
where $\mc{S}_\Sigma(V)\in\Sym(n)$ is the unique solution of the Sylvester equation $\Sigma\mc{S}_\Sigma(V)+\mc{S}_\Sigma(V)\Sigma=V$.
\end{definition}

\begin{theorem}[\chixthmbwgeodsymntitle]\label{thm:bw_geodesics_sym^+(n)}
Let $\Sigma\in\Sym^+(n)$.
\begin{enumerate}
    \item (Exponential map) \cite{Malago18} For all $V\in T_\Sigma\Sym^+(n)\equiv\Sym(n)$, the geodesic from $\Sigma$ with initial speed $V$ writes $\gamma_{(\Sigma,V)}(t)=\Sigma+tV+t^2\mc{S}_\Sigma(V)\Sigma\mc{S}_\Sigma(V)\in\Sym^+(n)$.
    
    \item (Definition interval) Let $\lambda_\maxi=\max\sp(\mc{S}_\Sigma(V))$ and $\lambda_\mini=\min\sp(\mc{S}_\Sigma(V))$. The definition interval of the geodesic $\gamma_{(\Sigma,V)}$ is the interval $I_{\Sigma,V}$ defined by:
    \begin{enumerate}[label=$\cdot$]
        \item $I_{\Sigma,V}=(-\frac{1}{\lambda_\maxi},-\frac{1}{\lambda_\mini})$ if $\lambda_\mini<0<\lambda_\maxi$,
        \item $I_{\Sigma,V}=(-\infty,-\frac{1}{\lambda_\mini})$ if $\lambda_\mini<0$ and $\lambda_\maxi\ls 0$,
        \item $I_{\Sigma,V}=(-\frac{1}{\lambda_\maxi},+\infty)$ if $\lambda_\mini\gs 0$ and $\lambda_\maxi>0$,
        \item $I_{\Sigma,V}=\R$ if $\lambda_\mini=\lambda_\maxi=0$ (which only happens for $V=0$).
    \end{enumerate}
    
    \item (Cut time) The cut time is $t_{cut}(\Sigma,V)=-\frac{1}{\lambda_\mini}$ if $\lambda_\mini<0$ or $+\infty$ otherwise. The geodesic $\gamma_{(\Sigma,V)}:I_{\Sigma,V}\lto\mc{M}$ is even minimizing on $I_{\Sigma,V}$.
    
    \item (Logarithm map) For all $\Lambda\in\Sym^+(n)$, there exists a unique preimage $V\in\mc{P}re_\Sigma(\Lambda)$. It writes $V=2\,\sym(\Sigma^{1/2}(\Sigma^{1/2}\Lambda\Sigma^{1/2})^{1/2}\Sigma^{-1/2})-2\Sigma$, where we denote $\sym(M)=\frac{1}{2}(M+M^\top)$. The geodesic joining $\Sigma$ to $\Lambda$ writes:
    \begin{equation}\label{eq:geodesic_full-rank}
        \gamma_{\Sigma\to\Lambda}(t)=(1-t)^2\Sigma+t^2\Lambda+2t(1-t)\,\sym(\Sigma^{1/2}(\Sigma^{1/2}\Lambda\Sigma^{1/2})^{1/2}\Sigma^{-1/2}).
    \end{equation}
    
    Moreover, it is a logarithm: $V\in\mc{L}og_x(y)$. Thus the logarithm map is defined on $\mc{U}_\Sigma=\Sym^+(n)$ and it writes:
    \begin{equation}
        \Log_\Sigma:\fun{\Sym^+(n)}{T_\Sigma\Sym^+(n)}{\Lambda}{2\sym(\Sigma^{1/2}(\Sigma^{1/2}\Lambda\Sigma^{1/2})^{1/2}\Sigma^{-1/2})-2\Sigma}.
    \end{equation}
\end{enumerate}
\begin{NDLR}
See the proof of Theorem \ref{thm:bw_geodesics_sym^+(n)} in Appendix \ref{subsec:proof:thm:bw_geodesics_sym^+(n)}.
\end{NDLR}
\end{theorem}

\begin{remark}
The minimizing geodesic $\gamma_{(\Sigma,V)}:I_{\Sigma,V}\lto\Sym^+(n)$ clearly has a limit at the finite boundaries of $I_{\Sigma,V}$. When $I_{\Sigma,V}$ is bounded, we can define $\Sigma_0=\lim_{t\to-1/\lambda_\maxi}\gamma_{(\Sigma,V)}(t)$, $\Sigma_1=\lim_{t\to-1/\lambda_\mini}\gamma_{(\Sigma,V)}(t)$ and the extended curve $\gamma:\bar{I}_{\Sigma,V}=[-1/\lambda_\maxi,-1/\lambda_\mini]$ by $\gamma(t)=\gamma_{(\Sigma,V)}(t)$ for $t\in I_{\Sigma,V}$, $\gamma(-1/\lambda_\maxi)=\Sigma_0$ and $\gamma(-1/\lambda_\mini)=\Sigma_1$. The curve $\gamma$ is a minimizing geodesic on $I_{\Sigma,V}$. Thus, by Lemma \ref{lemma:length} (continuity of the length), the curve $\gamma$ is a minimizing geodesic on $[-1/\lambda_\maxi,-1/\lambda_\mini]$. So we already have examples of minimizing geodesics between two degenerate matrices which pass through the principal stratum $\Sym^+(n)$. For instance, we have $\Sigma=\begin{pmatrix}4 & 0\\0 & 0\end{pmatrix}$ and $\Lambda=\begin{pmatrix}0 & 0\\0 & 4\end{pmatrix}$ which are linked by the geodesic $\Exp_{I_2}(tV)=\begin{pmatrix}(1+t)^2 & 0\\0 & (1-t)^2\end{pmatrix}$ for $t\in(-1,1)$ with $V=\begin{pmatrix}2 & 0\\0 & -2\end{pmatrix}$.

Moreover, from the viewpoint of the Levi-Civita connection, the curve $t\lmto\Exp_\Sigma(tV)$ is a geodesic (self-parallel curve) on each subinterval of the set $J_{\Sigma,V}=\{t\in\R|\Exp_\Sigma(tV)\in\Sym^+(n)\}=\R\backslash{\{t\in\R|-\frac{1}{t}\in\sp(\mc{S}_\Sigma(V))\}}$, which is $\R$ without a maximum of $n$ points. Theorem \ref{thm:bw_geodesics_sym^+(n)} actually states that every geodesic is minimizing on its domain. Hence, the minimizing geodesic $\gamma_{(\Sigma,V)}:I_{\Sigma,V}\lto\Sym^+(n)$ naturally extends to a curve $\gamma_{(\Sigma,V)}:\R\lto\Cov(n)$ which is a minimizing geodesic on the segments delimited by two consecutive values in $\R\backslash{J_{\Sigma,V}}$.
\end{remark}

\section{Geodesics of the Bures-Wasserstein metric on $\Sym^+(n,k)$}\label{sec:bures-wasserstein_sym^+(n,k)}

In this section, we give complements and new results on the Bures-Wasserstein geodesics on the manifold of PSD matrices of fixed rank $k$, $\Sym^+(n,k)$. They were mainly studied in \cite{Massart20}. The formulae of the exponential map and its definition domain were kept implicit because they were formulated in function of horizontal vectors in the total space $\R^{n\times k}_*$ of matrices of full-rank $k$. We compute the horizontal lift, which allows us to express the Bures-Wasserstein metric, the exponential map and the definition interval directly in function of the tangent vector. Moreover, the characterization of preimages of the exponential map given in \cite{Massart20} seems to omit the condition that a preimage $v\in\mc{P}re_x(y)$ must satisfy $1\in I_{x,v}$. Otherwise, the geodesic $\gamma_{x,v}:I_{x,v}\lto\mc{M}$ could leave the space before reaching $y$. Therefore, we characterize the set of preimages by relying on their work and considering this additional condition. We also give an explicit formula of the minimizing geodesic joining two points when it is unique and we specify the number of minimizing geodesics between two points otherwise. Finally, we compute the injection domain that was kept implicit in \cite{Massart20}. After Theorem \ref{thm:bw_geodesics_sym^+(n,k)}, we precisely specify the novelty of our result with respect to the reference work \cite{Massart20}. Then, we give examples to illustrate the possible cases for the number of preimages and logarithms.

\begin{definition}[Bures-Wasserstein metric on $\Sym^+(n,k)$]\cite{Massart20}
The Bures-Was\-serstein metric on $\Sym^+(n,k)$ is the quotient Riemannian metric induced by the submersion $\pi:X\in\R^{n\times k}_*\lmto XX^\top\in\Sym^+(n,k)$ and the Frobenius metric on $\R^{n\times k}$. Let $X\in\R^{n\times k}_*$ such that $XX^\top=\Sigma$ and let $V\in T_\Sigma\Sym^+(n,k)$. The vertical and horizontal spaces are:
\begin{enumerate}
    \itemsep0em
    \item (Vertical space) $\mc{V}_X=\ker d_X\pi=X\,\Skew(k)$,
    \item (Horizontal space) $\mc{H}_X=\{X(X^\top X)^{-1}F+X_\perp K,F\in\Sym(k),K\in\Mat(n-k,k)\}$ where $X_\perp\in\Mat(n,n-k)$ has orthonormal columns ($X_\perp^\top X_\perp=I_{n-k}$) that are orthogonal to the columns of $X$ ($X^\top X_\perp=0$).
\end{enumerate}
\end{definition}

\begin{theorem}[\chixthmhorliftmettitle]\label{thm:horlift_metric}
\chixthmhorliftmet
\begin{NDLR}
See the proof of Theorem \ref{thm:horlift_metric} in Appendix \ref{subsec:proof:thm:horlift_metric}.
\end{NDLR}
\end{theorem}

\begin{theorem}[\chixthmbwgeodsymnktitle]\label{thm:bw_geodesics_sym^+(n,k)}
~\\
Let $\Sigma,\Lambda\in\Sym^+(n,k)$ and $X,Y\in\R^{n\times k}_*$ such that $XX^\top=\Sigma$ and $YY^\top=\Lambda$. Let $U\in\St(n,k)$ and $D\in\Diag^+(k)$ such that $\Sigma=UDU^\top$.
\begin{enumerate}
    \item (Exponential map) For all $V\in T_\Sigma\Sym^+(n,k)$, the geodesic from $\Sigma$ with initial speed $V$ is $\gamma_{(\Sigma,V)}:t\in I_{\Sigma,V}\lmto\Sigma+tV+t^2W_{\Sigma,V}$, where $W_{\Sigma,V}=S_{\Sigma,V}\Sigma S_{\Sigma,V}+S_{\Sigma,V}V(I_n-UU^\top)+(I_n-UU^\top)VS_{\Sigma,V}+(I_n-UU^\top)V\Sigma^+ V(I_n-UU^\top)$ and $S_{\Sigma,V}=U\mc{S}_D(U^\top VU)U^\top$.
    
    \item (Definition interval) Let $F^0_{X,V}=\mc{S}_{X^\top X}((X^\top X)^{-1/2}X^\top VX(X^\top X)^{-1/2})$ and $M^0_{X,V}=(X^\top X)^{-3/2}X^\top V(I_n-X(X^\top X)^{-1}X^\top)VX(X^\top X)^{-3/2}\in\Sym(n)$. Let $\mc{E}_{\Sigma,V}=\{\lambda\in\sp(F^0_{X,V})|\,\ker(\lambda I_k-F^0_{X,V})\cap\ker(M^0_{X,V})\ne\{0\}\}\subseteq\sp(S_{\Sigma,V})$. If $\mc{E}_{\Sigma,V}$ is non-empty, then let $\lambda_+=\max\mc{E}_{\Sigma,V}$ and $\lambda_-=\min\mc{E}_{\Sigma,V}$. The definition interval of the geodesic $\gamma_{(\Sigma,V)}$ is the interval $I_{\Sigma,V}$ defined by:
    \begin{enumerate}[label=$\cdot$]
        \item $I_{\Sigma,V}=(-\frac{1}{\lambda_+},-\frac{1}{\lambda_-})$ if $\lambda_-<0<\lambda_+$,
        \item $I_{\Sigma,V}=(-\infty,-\frac{1}{\lambda_-})$ if $\lambda_-<0$ and $\lambda_+\ls 0$,
        \item $I_{\Sigma,V}=(-\frac{1}{\lambda_+},+\infty)$ if $\lambda_-\gs 0$ and $\lambda_+>0$,
        \item $I_{\Sigma,V}=\R$ if $\mc{E}_{\Sigma,V}$ is empty.
    \end{enumerate}
    
    Applying this to $X=UD^{1/2}$ without loss of generality, $F^0_{X,V}=\mc{S}_D(U^\top VU)$ and $M^0_{X,V}=D^{-1}U^\top V(I_n-UU^\top)VUD^{-1}$ which is a bit more tractable to compute $\mc{E}_{\Sigma,V}$.
    
    \item (Cut time) Let $\lambda_\maxi=\max\sp(F^0_{X,V})$ and $\lambda_\mini=\min\sp(F^0_{X,V})$. Note that if $\mc{E}_{\Sigma,V}\ne\emptyset$, then we have $(\lambda_-,\lambda_+)\subseteq(\lambda_\mini,\lambda_\maxi)$. The cut time is $t_{cut}(\Sigma,V)=-\frac{1}{\lambda_\mini}$ if $\lambda_\mini<0$ or $+\infty$ otherwise. Symmetrically, we have $t_{cut}(\Sigma,-V)=\frac{1}{\lambda_\maxi}$ if $\lambda_\maxi>0$ or $+\infty$ otherwise.
    
    \item (Preimages) We define the indexing set $\mc{I}^{\mc{P}re}_{X,Y}$ by:
    \begin{align*}
        \mc{I}^{\mc{P}re}_{X,Y}=\{R\in\Orth(n)|&H:=X^\top YR^\top\in\Sym(n)\mathrm{~and~}\\
        &\forall\mu<0,\ker(\mu I_k-(X^\top X)^{-1/2}H(X^\top X)^{-1/2})\\
        &\cap\ker(\mu^2 I_k-(X^\top X)^{-1/2}RY^\top YR^\top (X^\top X)^{-1/2})=\{0\}\}.
    \end{align*}
    
    For $R\in\mc{I}^{\mc{P}re}_{X,Y}$, we denote $H=H_{X,Y,R}=X^\top YR^\top$ so that $X^\top Y=HR$. Then the map $R\in\mc{I}^{\mc{P}re}_{X,Y}\lmto V=2\,\sym(XRY^\top)-2\Sigma\in\mc{P}re_\Sigma(\Lambda)$ is a bijection whose inverse is $V\in\mc{P}re_\Sigma(\Lambda)\lmto R=(Y^\top Y)^{-1}Y^\top(X+V^\#_X)\in\mc{I}^{\mc{P}re}_{X,Y}$.
    
    The geodesic joining $\Sigma$ to $\Lambda$ parametrized by $R\in\mc{I}^{\mc{P}re}_{X,Y}$ writes:
    \begin{equation}
        \forall t\in[0,1],\gamma_{\Sigma\to\Lambda}^R(t)=(1-t)^2\Sigma+t^2\Lambda+2t(1-t)\sym(XRY^\top).
    \end{equation}
    
    \item (Logarithms) Let $\mc{I}^{\mc{L}og}_{X,Y}=\{R\in\Orth(n)|H_{X,Y,R}=X^\top YR^\top\in\Cov(n)\}=\{R\in\Orth(n)|H_{X,Y,R}=(X^\top\Lambda X)^{1/2}\}=\{R\in\Orth(n)|\,X^\top Y=(X^\top\Lambda X)^{1/2}R\}\subseteq\mc{I}^{\mc{P}re}_{X,Y}$.
    
    Then, the map $R\in\mc{I}^{\mc{L}og}_{X,Y}\lmto V=2\,\sym(XRY^\top)-2\Sigma\in\mc{L}og_\Sigma(\Lambda)$ is a bijection whose inverse is $V\in\mc{L}og_\Sigma(\Lambda)\lmto R=(Y^\top Y)^{-1}Y^\top(X+V^\#_X)\in\mc{I}^{\mc{L}og}_{X,Y}$.
    
    \item (Logarithm map) Let $r=\rk(\Sigma\Lambda)=\rk(X^\top Y)=\rk(H)$.
    \begin{enumerate}
        \item If $r=k$, then there exists a unique logarithm of $\Lambda$ from $\Sigma$. In this case, the minimizing geodesic joining $\Sigma$ to $\Lambda$ writes:
        \small
        \begin{equation}
            \gamma_{\Sigma\to\Lambda}(t)=(1-t)^2\Sigma+t^2\Lambda+2t(1-t)\sym(\Sigma^{1/2}((\Sigma^{1/2}\Lambda\Sigma^{1/2})^{1/2})^-\Sigma^{1/2}\Lambda).
        \end{equation}
        \normalsize
        \item If $r=k-1$, then there exist exactly two logarithms of $\Lambda$ from $\Sigma$.
        
        \item If $r<k-1$, then there is an infinity of logarithms of $\Lambda$ from $\Sigma$.
    \end{enumerate}
    Therefore, the logarithm map is defined on $\mc{U}_\Sigma=\{\Lambda\in\Sym^+(n,k)|\rk(\Sigma\Lambda)=k\}$ and it writes $\Log_\Sigma:\Lambda\in\mc{U}_\Sigma\lmto 2\,\sym(\Sigma^{1/2}((\Sigma^{1/2}\Lambda\Sigma^{1/2})^{1/2})^-\Sigma^{1/2}\Lambda)-2\Sigma\in T_\Sigma\Sym^+(n,k)$.
\end{enumerate}
\begin{NDLR}
See the proof of Theorem \ref{thm:bw_geodesics_sym^+(n,k)} in Appendix \ref{subsec:proof:thm:bw_geodesics_sym^+(n,k)}.
\end{NDLR}
\end{theorem}

\begin{remark}
Let us clarify our contributions in Theorem \ref{thm:bw_geodesics_sym^+(n,k)} with respect to the reference paper \cite{Massart20}.
\begin{enumerate}
    \itemsep0em
    \item (Exponential map) The formula is new. Only the exponential map of a horizontal vector in the total space $\R^{n\times k}_*$ was given in \cite{Massart20}.
    \item (Definition interval) The definition interval was formulated in the total space $\R^{n\times k}_*$ in \cite{Massart20}. The novelty here is to formulate it in function of $V$ thanks to the horizontal lift.
    \item (Cut time) This is new.
    \item (Preimages) Note that $\mc{I}^{\mc{P}re}_{X,Y}\subseteq\mc{I}^{\mc{S}ol}_{X,Y}:=\{R\in\Orth(n)|X^\top YR^\top\in\Sym(n)\}$. The set $\mc{I}^{\mc{S}ol}_{X,Y}$ indexes the solutions of the logarithm equation $\Exp_\Sigma(V)=\Lambda$ witout taking into account the condition $1\in I_{\Sigma,V}$. The fact that there may exist $R\in\mc{I}^{\mc{S}ol}_{X,Y}$ that does not lead to a preimage seems to be forgotten in \cite{Massart20}. Indeed, the curve ``$\gamma^R$" may hit the boundary before reaching $\Lambda$. Therefore, among these candidate $R$'s such that $X^\top YR^\top\in\Sym(n)$, we specify the set of $R$'s that really define a geodesic from $\Sigma$ to $\Lambda$ in the manifold $\Sym^+(n,k)$, based on the condition underlying the definition interval.
    \item (Logarithms) It was stated in \cite{Massart20} that if $R$ leads to a logarithm then $H\gs 0$, and that if $H\gs 0$, then $R$ leads to a preimage which is additionally a logarithm. However, it is not stated clearly that in this case, $H$ has to be equal to $(X^\top\Lambda X)^{1/2}$. It is important for the next point though. The expression of the logarithms is new, although very straightforward. It was not given in \cite{Massart20} because they prefer to work in the total space $\R^{n\times k}_*$.
    \item (Logarithm map) It was stated that the logarithm is unique if and only if $r=k$ in \cite{Massart20}. However, it was not stated that there are exactly two logarithms when $r=k-1$ and that there is an infinity of \textit{logarithms} when $r<k-1$. The expression of the minimizing geodesic when it is unique is also new.
\end{enumerate}
In other words, we have a minor contribution on the reformulation in $\Sym^+(n,k)$ of results stated in the total space, and more important contributions on the injection domain (cut time), the expression of the minimizing geodesic when it is unique and the clarification between the three sets $\mc{I}^{\mc{L}og}_{X,Y}\subseteq\mc{I}^{\mc{P}re}_{X,Y}\subseteq\mc{I}^{\mc{S}ol}_{X,Y}:=\{R\in\Orth(n)|X^\top YR^\top\in\Sym(n)\}$. We give several examples below to illustrate the differences between these three sets.
\end{remark}

\begin{remark}
The definitions of the sets $\mc{E}_{\Sigma,V}$ and $\mc{I}^{\mc{P}re}_{\Sigma,V}$ might seem intricate. It is difficult to simplify them though. Nevertheless, the set $\mc{E}_{\Sigma,V}$ is easy to determine numerically. Analogously, it is easy to determine numerically if a candidate $R\in\mc{I}^{\mc{S}ol}_{X,Y}$ given in \cite[Lemma 4.1]{Massart20} belongs to $\mc{I}^{\mc{P}re}_{X,Y}$.
\end{remark}

\begin{examples}
Let us denote $R_\theta^-=\left(\begin{smallmatrix}-\cos\theta & \sin\theta\\\sin\theta & \cos\theta\end{smallmatrix}\right)\in\Orth(2)\cap\Sym(2)$.
\begin{enumerate}
    \item Let $\Sigma=\left(\begin{smallmatrix}1 & 0 & 0\\0 & 1 & 0\\0 & 0 & 0\end{smallmatrix}\right)$ and $\Lambda=\left(\begin{smallmatrix}4 & 0 & 0\\0 & 4 & 0\\0 & 0 & 0\end{smallmatrix}\right)=4\Sigma$ in $\Sym^+(3,2)$ with $r=2$. Then, let $X=\left(\begin{smallmatrix}1 & 0\\0 & 1\\0 & 0\end{smallmatrix}\right)$ and $Y=\left(\begin{smallmatrix}2 & 0\\0 & 2\\0 & 0\end{smallmatrix}\right)$. Then $XX^\top=\Sigma$, $YY^\top=\Lambda$ and $X^\top Y=2I_2$. Thus the candidate pairs for $(H,R)$ are $(2I_2,I_2)$, $(-2I_2,-I_2)$ and $(2R_\theta^-,R_\theta^-)$. One can show that:
    \begin{enumerate}
        \item $R=I_2$ leads to the unique minimizing geodesic $\gamma_{\Sigma\to\Lambda}(t)=(1+t)^2\Sigma\in\Sym^+(3,2)$ for $t\in(-1,+\infty)\supset[0,1]$,
        \item there is no non-minimizing geodesic,
        \item $R=-I_2$ or $R=R_\theta^-$ lead to curves that hit $\Sym^+(3,1)$ at $t=\frac{1}{3}<1$, e.g. $R=-I_2$ leads to the curve $\gamma(t)=(1-3t)^2\Sigma\in\Sym^+(3,2)$ only for $t\in(-\infty,\frac{1}{3})$.
    \end{enumerate}
    
    \item Let $\Sigma=\left(\begin{smallmatrix}1 & 0 & 0\\0 & 1 & 0\\0 & 0 & 0\end{smallmatrix}\right)$ and $\Lambda=\left(\begin{smallmatrix}1 & 0 & 0\\0 & 1 & 1\\0 & 1 & 1\end{smallmatrix}\right)$ in $\Sym^+(3,2)$ with $r=2$ again. Then, let $X=\left(\begin{smallmatrix}1 & 0\\0 & 1\\0 & 0\end{smallmatrix}\right)$ and $Y=\left(\begin{smallmatrix}1 & 0\\0 & 1\\0 & 1\end{smallmatrix}\right)$. Then $X^\top Y=I_2$ so the candidate pairs for $(H,R)$ are $(I_2,I_2)$, $(-I_2,-I_2)$ and $(R_\theta^-,R_\theta^-)$. One can show that:
    \begin{enumerate}
        \item $R=I_2$ leads to the geodesic $\gamma^{I_2}_{\Sigma\to\Lambda}(t)=\left(\begin{smallmatrix}1 & 0 & 0\\0 & 1 & t\\0 & t & t^2\end{smallmatrix}\right)\in\Sym^+(3,2)$ for $t\in\R$ which is minimizing on $\left[-\frac{2}{1+\sqrt{2}},\frac{2}{\sqrt{2}-1}\right]\supset[0,1]$,
        \item $R=R_\theta^-$ for $\theta\ne 0$ lead to non-minimizing geodesics, e.g. $R=R_\pi^-$ leads to the curve $\gamma^{R_\pi^-}_{\Sigma\to\Lambda}(t)=\left(\begin{smallmatrix}1 & 0 & 0\\0 & (1-2t)^2 & -t(1-2t)\\0 & -t(1-2t) & t^2\end{smallmatrix}\right)\in\Sym^+(3,2)$ for $t\in[0,1]$,
        \item $R=R_0^-$ and $R=-I_2$ lead to curves that hit $\Sym^+(3,1)$ at $t=\frac{1}{2}<1$, e.g. $R=R_0^-$ leads to the curve $\gamma(t)=\left(\begin{smallmatrix}(1-2t)^2 & 0 & 0\\ 0 & 1 & t\\ 0 & t & t^2\end{smallmatrix}\right)\in\Sym^+(3,2)$ only for $t\in(\infty,\frac{1}{2})$.
    \end{enumerate}
    
    \item Let $\Sigma=\left(\begin{smallmatrix}1 & 0 & 0\\0 & 1 & 0\\0 & 0 & 0\end{smallmatrix}\right)$ and $\Lambda=\left(\begin{smallmatrix}1 & 0 & 0\\0 & 0 & 0\\0 & 0 & 1\end{smallmatrix}\right)$ in $\Sym^+(3,2)$ with $r=1$. Then, let $X=\left(\begin{smallmatrix}1 & 0\\0 & 1\\0 & 0\end{smallmatrix}\right)$ and $Y=\left(\begin{smallmatrix}1 & 0\\0 & 0\\0 & 1\end{smallmatrix}\right)$. Then $X^\top Y=\left(\begin{smallmatrix}1 & 0\\0 & 0\end{smallmatrix}\right)$ so the candidate values of $R$ are $\Diag(\pm 1,\pm 1)$. One can show that:
    \begin{enumerate}
        \item $R^0_\pm=\Diag(1,\pm 1)$ lead to two minimizing geodesics whose expressions are $\gamma^{R^0_\pm}_{\Sigma\to\Lambda}(t)=\left(\begin{smallmatrix}1 & 0 & 0\\0 & (1-t)^2 & \pm t(1-t)\\0 & \pm t(1-t) & t^2\end{smallmatrix}\right)\in\Sym^+(3,2)$ for $t\in[0,1]$,
        \item there is no non-minimizing geodesic,
        \item $R^1_\pm=\Diag(-1,\pm 1)$ lead to curves that hit $\Sym^+(3,1)$ at $t=\frac{1}{2}<1$, namely $\gamma^{R^1_\pm}(t)=\left(\begin{smallmatrix}(1-2t)^2 & 0 & 0\\0 & (1-t)^2 & \pm t(1-t)\\0 & \pm t(1-t) & t^2\end{smallmatrix}\right)\in\Sym^+(3,2)$ only for $t\in(-\infty,\frac{1}{2})$.
    \end{enumerate}
    
    \item Let $\Sigma=\left(\begin{smallmatrix}I_2 & 0\\0 & 0\end{smallmatrix}\right)$ and $\Lambda=\left(\begin{smallmatrix}0 & 0\\0 & I_2\end{smallmatrix}\right)$ in $\Sym^+(4,2)$ with $r=0$. Then, let $X=\left(\begin{smallmatrix}I_2\\0\end{smallmatrix}\right)$ and $Y=\left(\begin{smallmatrix}0\\ I_2\end{smallmatrix}\right)$. Then $X^\top Y=0$ so every $R\in\Orth(2)$ is a candidate. One can show that any $R\in\Orth(2)$ leads to a minimizing geodesic $\gamma_{\Sigma\to\Lambda}^R(t)=\left(\begin{smallmatrix}(1-t)^2I_2 & t(1-t)R\\t(1-t)R^\top & t^2I_2\end{smallmatrix}\right)\in\Sym^+(4,2)$ for $t\in\R$.
\end{enumerate}
\end{examples}

In the two last sections, we studied the geodesics and the minimizing geodesics within each stratum. In the next section, we turn to the study of the minimizing geodesic segments in the Bures-Wasserstein metric space $(\Cov(n),d^{\mathrm{BW}})$, that is between any two covariance matrices of any rank.

\section{Minimizing geodesics of the Bures-Wasserstein distance on $\Cov(n)$}\label{sec:bures-wasserstein_cov(n)}

In this section, we completely characterize the Bures-Wasserstein minimizing geodesic segments between any two covariance matrices. We show that they have constant rank on the interior of the segment and we give an explicit expression. Moreover, we show that the number of geodesics depends on the ranks of the extremities and we give this number in all cases. More precisely, we show that minimizing geodesics between $\Sigma$ and $\Lambda\in\Cov(n)$ are parametrized by the closed unit ball of $\R^{(k-r)\times(l-r)}$ for the spectral norm, where $k,l,r$ are the respective ranks of $\Sigma,\Lambda,\Sigma\Lambda$. We also give the number of geodesics of minimal rank. Finally, we show that there exists a canonical geodesic with an expression that does not depend on the ranks of the extremities. This expression coincides with the formula in low rank when the minimizing geodesic is unique and with the formula in full rank. The proofs are deferred to the Supplementary Material.

\subsection{Characterization of minimizing geodesics}\label{subsec:bures-wasserstein_characterization_geodesics}

The following lemma states that the rank of a minimizing geodesic segment is constant on the interior of the segment. Then, Theorem \ref{thm:minimizing_geodesics} characterizes the Bures-Wasserstein minimizing geodesic segments.

\begin{lemma}[Rank of minimizing curve]\label{lemma:rank}
Let $\gamma:[0,1]\lto\Cov(n)$ be a minimizing curve from $\Sigma$ to $\Lambda$. Then $\gamma$ has constant rank $p\gs\max(\rk(\Sigma),\rk(\Lambda))$ on $(0,1)$.
\end{lemma}

\begin{proof}
Let $p=\max_{t\in[0,1]}\rk(\gamma(t))$ and let $t_0\in[0,1]$ such that $\rk(t_0)=p$. By Lemma \ref{lemma:minimizing_geodesics_strata}, for all $t\in(0,t_0)\cup(t_0,1)$, $(H_{\rk(\gamma(t))})\ls H_{\rk(\gamma(t_0))}$ so $\rk(\gamma(t))\gs\rk(\gamma(t_0))=p$ so $\rk(\gamma(t))=p$.
\end{proof}

\begin{theorem}[\chixthmmingeodtitle]\label{thm:minimizing_geodesics}
\chixthmmingeod\\
\begin{NDLR}
See the proof of Theorem \ref{thm:minimizing_geodesics} in Appendix \ref{subsec:proof:thm:minimizing_geodesics}.
\end{NDLR}
\end{theorem}

\subsection{Number of minimizing geodesics}\label{subsec:bures-wasserstein_number_geodesics}

In this section, we count the number of minimizing geodesic segments between two covariance matrices. We start with an elementary lemma.

\begin{lemma}[\chixlemmaelemalgtitle]\label{lemma:bures-wasserstein_elem_algebra}
\chixlemmaelemalg
\begin{NDLR}
See the proof of Lemma \ref{lemma:bures-wasserstein_elem_algebra} in Appendix \ref{subsec:proof:lemma:bures-wasserstein_elem_algebra}.
\end{NDLR}
\end{lemma}

\begin{theorem}[\chixthmnbbwgeodtitle]\label{thm:number_bw_minimizing_geodesics}
Let $\Sigma,\Lambda\in\Cov(n)$ with $\rk(\Sigma)=k$ and $\rk(\Lambda)=l$. We assume that $k\gs l$ without loss of generality. We denote $r=\rk(\Sigma\Lambda)$. We have $l-r\ls n-k$.
\begin{enumerate}
    \itemsep0em
    \item There exists a bijection between the set of minimizing geodesics from $\Sigma$ to $\Lambda$ and the closed unit ball of $\R^{(k-r)\times(l-r)}$ for the spectral norm $\bar{\mc{B}}_{\mathrm{S}}(0,1)=\{R_0\in\R^{(k-r)\times(l-r)}|\,\|R_0\|_{\mathrm{S}}\ls 1\}=\{R_0\in\R^{(k-r)\times(l-r)}|\,0\ls R_0^\top R_0\ls I_{l-r}\}$.
    \item The minimizing geodesic is unique if and only if $r=l$. This includes the cases $k=n$.
    \item There is an infinite number of minimizing geodesics if and only if $r<l$.
    \item The minimizing geodesics corresponding to the choices $R_0\in\St(k-r,l-r)$ (including the empty matrix if $r=l$) have rank exactly $k$ on $[0,1)$ (on $[0,1]$ if $l=k$). Note that $\St(k-r,l-r)$ is included in the unit sphere $\mc{S}_{\mathrm{S}}(0,1)=\{R_0\in\R^{(k-r)\times(l-r)}|\,\|R_0\|_{\mathrm{S}}=1\}$.
    \item The minimizing geodesic corresponding to the choice $R_0=0$ (or the empty matrix if $r=l$) writes for all $t\in[0,1]$:
    \small
    \begin{equation}
        \gamma^0_{\Sigma\to\Lambda}(t)=(1-t)^2\Sigma+t^2\Lambda+2t(1-t)\,\sym(\Sigma^{1/2}((\Sigma^{1/2}\Lambda\Sigma^{1/2})^{1/2})^-\Sigma^{1/2}\Lambda).
    \end{equation}
    \normalsize
    If $r=l$, it has rank exactly $k$ on $[0,1)$.
\end{enumerate}
The number of minimizing geodesic segments in $\Sym^+(n,k)$ and in $\Cov(n)$ is summarized in Table \ref{tab:minimizing_geodesics} with $n\gs k\gs l\gs r$.\\
\begin{NDLR}
See the proof of Theorem \ref{thm:number_bw_minimizing_geodesics} in Appendix \ref{subsec:proof:thm:number_bw_minimizing_geodesics}.
\end{NDLR}
\end{theorem}

\begin{table}[h]
    \centering
    \begin{tabular}{|c|c|c|c|c|}
        \hline
        \multirow{2}{*}{$\Sigma\in$} & \multirow{2}{*}{$\Lambda\in$} & \multirow{2}{*}{$r=\rk(\Sigma\Lambda)$} & \multicolumn{2}{c|}{Number of minimizing geodesics}\\
        \cline{4-5}
         &  &  & $~\mathrm{in~}\Sym^+(n,k)~$ & $\mathrm{in~}\Cov(n)$\\
        \hline
        $\Sym^+(n)$ & $\Sym^+(n)$ & $n$ & $1$ & $1$\\
        \hline
        $\Sym^+(n)$ & $\Sym^+(n,k)$ & $k$ & $1$ & $1$\\
        \hline
        \multirow{3}{*}{$\Sym^+(n,k)$} & \multirow{3}{*}{$\Sym^+(n,k)$} & $k$ & $1$ & $1$\\
        \cline{3-5}
        & & $k-1$ & $2$ & $\infty$\\
        \cline{3-5}
        & & $<k-1$ & $\infty$ & $\infty$\\
        \hline
        \multirow{2}{*}{$\Sym^+(n,k)$} & \multirow{2}{*}{$\Sym^+(n,l)$} & $l$ & $1$ & $1$\\
        \cline{3-5}
        & & $<l$ & $\infty$ & $\infty$\\
        \hline
    \end{tabular}
    \caption{Number of Bures-Wasserstein minimizing geodesics ($n\gs k\gs l\gs r$).}
    \label{tab:minimizing_geodesics}
\end{table}

\section{Conclusion}\label{sec:bures-wasserstein_conclusion}

We have answered several open questions on geodesics of the Bures-Wasserstein distance on covariance matrices. Beyond geodesics, a very important element of Riemannian geometry is the curvature. We know that the space of covariance matrices with the Bures-Wasserstein distance is an Aleksandrov space of non-negative curvature \cite{Takatsu11} and we know the curvature tensor in each stratum \cite{Takatsu10,Takatsu11,Massart19}. However, we lack a comprehensive and global approach of the curvature of the whole metric space. In particular, what is the appropriate notion of curvature to use to go from a stratum to another?

In the community of geometric statistics, most of the stratified spaces that were studied from the viewpoint of geodesics or curvature are very singular (spiders, trees) or a bit complex to start with (BHV space, Wald space) \cite{Feragen20}. Thus, the familiar example of the Bures-Wasserstein Riemannian orbit space appears to be a good basis to generalize concepts defined in Riemannian statistics. Indeed, after studying the geometry of these non-Riemannian spaces, what statistical tools should we define on them to generalize the Euclidean and Riemannian ones? This is probably the main question to investigate for the future.

\section*{Acknowledgements}

This project has received funding from the European Research Council (ERC) under the European Union’s Horizon 2020 research and innovation program (grant G-Statistics agreement No 786854). This work has been supported by the French government, through the UCAJEDI Investments in the Future project managed by the National Research Agency (ANR) with the reference number ANR-15-IDEX-01 and through the 3IA Côte d’Azur Investments in the Future project managed by the National Research Agency (ANR) with the reference number ANR-19-P3IA-0002. The authors warmly thank Anna Calissano for insightful discussions about the Bures-Wasserstein distance and orbit spaces.

\begin{appendices}

\section{Appendix}

In this Appendix, we prove the main results of the paper. For readability, we recall the result before stating the proof.

\subsection{Lemma \ref{lemma:bures-wasserstein_topologies}}\label{subsec:proof:lemma:bures-wasserstein_topologies}

(\chixlemmaeucbwtoptitle) \chixlemmaeucbwtop

\begin{proof}[Proof of Lemma \ref{lemma:bures-wasserstein_topologies}]
The map $\pi:X\in\Mat(n)\lmto XX^\top\in(\Cov(n),d^{\mathrm{E}})$ is continuous so the quotient topology, i.e. the topology induced by the Bures-Wasserstein distance, is finer than the Euclidean topology. Conversely, let $\mc{U}$ be an open set for the Bures-Wasserstein distance. Let $\Sigma\in\mc{U}$. Let $\varepsilon>0$ such that the Bures-Wasserstein ball $\mc{B}^{\mathrm{BW}}(\Sigma,\varepsilon)$ is included in $\mc{U}$. The set $\mc{V}=\pow_2(\mc{B}^{\mathrm{E}}(\Sigma^{1/2},\varepsilon))$ is open for the Euclidean distance because the map $\pow_2:\Sigma\lmto\Sigma^2$ is a homeomorphism of $(\Cov(n),d^{\mathrm{E}})$. Moreover, if $\Lambda\in\mc{V}$, then $d^{\mathrm{BW}}(\Sigma,\Lambda)\ls d^{\mathrm{E}}(\Sigma^{1/2},\Lambda^{1/2})\ls\varepsilon$ so $\Lambda\in\mc{B}^{\mathrm{BW}}(\Sigma,\varepsilon)\subseteq\mc{U}$. So $\mc{V}\subseteq\mc{U}$ is a Euclidean neighborhood of $\Sigma$, so $\mc{U}$ is open for the Euclidean distance. Therefore the two topologies coincide.
\end{proof}

\subsection{Theorem \ref{thm:bw_geodesics_sym^+(n)}}\label{subsec:proof:thm:bw_geodesics_sym^+(n)}

(\chixthmbwgeodsymntitle)\\
Let $\Sigma\in\Sym^+(n)$.
\begin{enumerate}
    \item (Exponential map) \cite{Malago18} For all $V\in T_\Sigma\Sym^+(n)\equiv\Sym(n)$, the geodesic from $\Sigma$ with initial speed $V$ writes $\gamma_{(\Sigma,V)}(t)=\Sigma+tV+t^2\mc{S}_\Sigma(V)\Sigma\mc{S}_\Sigma(V)\in\Sym^+(n)$.
    
    \item (Definition interval) Let $\lambda_\maxi=\max\sp(\mc{S}_\Sigma(V))$ and $\lambda_\mini=\min\sp(\mc{S}_\Sigma(V))$. The definition interval of the geodesic $\gamma_{(\Sigma,V)}$ is the interval $I_{\Sigma,V}$ defined by:
    \begin{enumerate}[label=$\cdot$]
        \item $I_{\Sigma,V}=(-\frac{1}{\lambda_\maxi},-\frac{1}{\lambda_\mini})$ if $\lambda_\mini<0<\lambda_\maxi$,
        \item $I_{\Sigma,V}=(-\infty,-\frac{1}{\lambda_\mini})$ if $\lambda_\mini<0$ and $\lambda_\maxi\ls 0$,
        \item $I_{\Sigma,V}=(-\frac{1}{\lambda_\maxi},+\infty)$ if $\lambda_\mini\gs 0$ and $\lambda_\maxi>0$,
        \item $I_{\Sigma,V}=\R$ if $\lambda_\mini=\lambda_\maxi=0$ (which only happens for $V=0$).
    \end{enumerate}
    
    \item (Cut time) The cut time is $t_{cut}(\Sigma,V)=-\frac{1}{\lambda_\mini}$ if $\lambda_\mini<0$ or $+\infty$ otherwise. The geodesic $\gamma_{(\Sigma,V)}:I_{\Sigma,V}\lto\mc{M}$ is even minimizing on $I_{\Sigma,V}$.
    
    \item (Logarithm map) For all $\Lambda\in\Sym^+(n)$, there exists a unique preimage $V\in\mc{P}re_\Sigma(\Lambda)$. It writes $V=2\,\sym(\Sigma^{1/2}(\Sigma^{1/2}\Lambda\Sigma^{1/2})^{1/2}\Sigma^{-1/2})-2\Sigma$, where we denote $\sym(M)=\frac{1}{2}(M+M^\top)$. The geodesic joining $\Sigma$ to $\Lambda$ writes:
    \begin{equation*}
        \gamma_{\Sigma\to\Lambda}(t)=(1-t)^2\Sigma+t^2\Lambda+2t(1-t)\,\sym(\Sigma^{1/2}(\Sigma^{1/2}\Lambda\Sigma^{1/2})^{1/2}\Sigma^{-1/2}).
    \end{equation*}
    
    Moreover, it is a logarithm: $V\in\mc{L}og_x(y)$. Thus the logarithm map is defined on $\mc{U}_\Sigma=\Sym^+(n)$ and it writes:
    \begin{equation*}
        \Log_\Sigma:\fun{\Sym^+(n)}{T_\Sigma\Sym^+(n)}{\Lambda}{2\sym(\Sigma^{1/2}(\Sigma^{1/2}\Lambda\Sigma^{1/2})^{1/2}\Sigma^{-1/2})-2\Sigma}.
    \end{equation*}
\end{enumerate}

\begin{proof}[Proof of Theorem \ref{thm:bw_geodesics_sym^+(n)}]
We prove statement 3 in the end because it requires statement 4.
\begin{enumerate}
    \item (Exponential map) The expression of the exponential map comes from \cite{Malago18}.
    
    \item (Definition domain) The domain $I_{\Sigma,V}$ is described in \cite{Malago18} as the connected component of $0$ in $J_{\Sigma,V}=\{t\in\R|I_n+t\mc{S}_\Sigma(V)\in\Sym^+(n)\}$. Since $t\in J_{\Sigma,V}$ if and only if $0\notin\{1+t\lambda|\lambda\in\sp(\mc{S}_\Sigma(V))\}$ if and only if $t\notin\{-\frac{1}{\lambda}|\lambda\in\sp(\mc{S}_\Sigma(V))\}$, we have $\max(-\infty,0]\cap\{-\frac{1}{\lambda}|\lambda\in\sp(\mc{S}_\Sigma(V))\}=-\frac{1}{\lambda_\maxi}$ if $\lambda_\maxi>0$ and $\min[0,+\infty)\cap\{-\frac{1}{\lambda}|\lambda\in\sp(\mc{S}_\Sigma(V))\}=-\frac{1}{\lambda_\mini}$ if $\lambda_\mini<0$. Therefore, we have the following cases:
    \begin{enumerate}[label=$\cdot$]
        \item if $\lambda_\mini<0<\lambda_\maxi$, then $I_{\Sigma,V}=(-\frac{1}{\lambda_\maxi},-\frac{1}{\lambda_\mini})$,
        \item if $\lambda_\mini<0$ and $\lambda_\maxi\ls0$, then $(-\infty,0]\subseteq J_{\Sigma,V}$ so $I_{\Sigma,V}=(-\infty,-\frac{1}{\lambda_\mini})$,
        \item if $\lambda_\mini\gs0$ and $\lambda_\maxi>0$, then $[0,+\infty)\subseteq J_{\Sigma,V}$ so $I_{\Sigma,V}=(-\frac{1}{\lambda_\maxi},+\infty)$,
        \item if $\lambda_\mini\gs 0$ and $\lambda_\maxi\ls 0$, which means $\lambda_\mini=\lambda_\maxi=0$, then $V=0$ and $I_{\Sigma,V}=J_{\Sigma,V}=\R$.
    \end{enumerate}

    \item[4.] (Logarithm map) The existence of a preimage $V\in\mc{P}re_\Sigma(V)$ and even a logarithm $V\in\mc{L}og_\Sigma(V)$ (because it satisfies $\|V\|=d(\Sigma,\Lambda)$) is due to \cite{Bhatia19}. The geodesic joining $\Sigma$ to $\Lambda$ (Equation \ref{eq:geodesic_full-rank}) is derived in \cite{Bhatia19} and it suffices to derive the expression at $t=0$ to compute $V=\dot\gamma_{\Sigma\to\Lambda}(0)=2\,\sym(\Sigma^{1/2}(\Sigma^{1/2}\Lambda\Sigma^{1/2})^{1/2}\Sigma^{-1/2})-2\Sigma$. 

    The uniqueness of the preimage comes from \cite[Proposition 4.4]{Massart20}. Indeed, it is stated that there exists a unique $W\in\mc{H}_{\Sigma^{1/2}}$ such that:
    \begin{enumerate}
        \item for all $t\in[0,1]$, $\Sigma^{1/2}(\Sigma^{1/2}+tW)\in\GL(n)$,
        \item $(\Sigma^{1/2}+W)(\Sigma^{1/2}+W)^\top=\Lambda$.
    \end{enumerate}
    Therefore, there exists a unique $V=d_{\Sigma^{1/2}}\pi(W)=\Sigma^{1/2}W^\top+W\Sigma^{1/2}\in T_\Sigma\Sym^+(n)$ (and $W=V^\#_{\Sigma^{1/2}}$) such that $1\in I_{\Sigma,V}$ (that is, for all $t\in[0,1]$, $\Exp_\Sigma(tV)\in\Sym^+(n)$) and $\Exp_\Sigma(V)=\Lambda$, i.e. $V\in\mc{P}re_\Sigma(\Lambda)$. Thus the logarithm map $\Log_\Sigma$ is defined on $\mc{U}_\Sigma=\Sym^+(n)$.

    \item[3.] (Cut time) Let us prove that $\gamma_{(\Sigma,V)}$ is minimizing on $I_{\Sigma,V}$. This will prove in particular that $t_{cut}(\Sigma,V)=\sup I_{\Sigma,V}$. Let $t,t'\in I_{\Sigma,V}$, $t<0<t'$, let $\Lambda=\Exp_\Sigma(tV)$ and $\Lambda'=\Exp_\Sigma(t'V)$. Changing the base point of the geodesic, we have $\Lambda'=\Exp_\Lambda((t'-t)V')$ with $V'=-\dot\gamma_{(\Sigma,V)}(t)\in T_\Lambda\Sym^+(n)$. Since for all $s\in[0,1]$, $(1-s)t+st'\in I_{\Sigma,V}$ and $\Exp_\Lambda(s(t'-t)V')=\Exp_\Sigma(((1-s)t+st')V)\in\Sym^+(n)$, we have $1\in I_{\Lambda,(t'-t)V'}$ so $(t'-t)V'\in\mc{P}re_\Lambda(\Lambda')$. By uniqueness of the preimage of $\Lambda'$ from $\Lambda$, $\Log_\Lambda(\Lambda')=(t'-t)V'$ and $\gamma_{\Lambda,(t'-t)V'}$ is minimizing on $[0,1]$. Equivalently, $\gamma_{(\Sigma,V)}$ is minimizing on $[t,t']$ so it is minimizing on $I_{\Sigma,V}$.
\end{enumerate}
\end{proof}

\subsection{Theorem \ref{thm:horlift_metric}}\label{subsec:proof:thm:horlift_metric}

(\chixthmhorliftmettitle)\\
\chixthmhorliftmet

\begin{proof}[Proof of Theorem \ref{thm:horlift_metric}]
\begin{enumerate}
    \item[1\&2.] We prove the expression of the tangent space and the horizontal lift together. Let $V\in T_\Sigma\Sym^+(n,k)$. The horizontal lift is defined by:
    \begin{enumerate}[label=$\cdot$]
        \itemsep0em
        \item (lift) $V=d_X\pi(V^\#)=X(V^\#)^\top+V^\# X^\top$,
        \item (horizontal) $V^\#=X(X^\top X)^{-1}F+X_\perp K$ where $F\in\Sym(r)$ and $K\in\Mat(n-r,r)$.
    \end{enumerate}
    When we plug the second equality in the first one and we multiply by $X^\top$ on the left and $X$ on the right, since $X^\top X_\perp=0$, we get immediately $X^\top VX=X^\top XF+FX^\top X$ so $F=\mc{S}_{X^\top X}(X^\top VX)$. By multiplying by $X_\perp^\top$ on the left instead, we get $X_\perp^\top VX=KX^\top X$ so $K=X_\perp^\top VX(X^\top X)^{-1}$. Since the matrix $(X(X^\top X)^{-1/2}\,;\,X_\perp)$ is orthogonal, we have $X(X^\top X)^{-1}X^\top+X_\perp X_\perp^\top=I_n$.
    
    We compute $d_X\pi(V^\#)$ to check that it is equal to $V$:
    \begin{align*}
        d_X\pi(V^\#)&=X[(X^\top X)^{-1}F+F(X^\top X)^{-1}]X^\top\\
        &\quad+2\,\sym(X_\perp X_\perp^\top VX(X^\top X)^{-1}X^\top)\\
        &=X(X^\top X)^{-1}X^\top VX(X^\top X)^{-1}X^\top\\
        &\quad+2\,\sym(X_\perp X_\perp^\top V(I_n-X_\perp X_\perp^\top))\\
        &=(I_n-X_\perp X_\perp^\top)V(I_n-X_\perp X_\perp^\top)\\
        &\quad+X_\perp X_\perp^\top V+VX_\perp X_\perp^\top-2X_\perp X_\perp^\top VX_\perp X_\perp^\top\\
        &=V-X_\perp X_\perp^\top VX_\perp X_\perp^\top.
    \end{align*}
    Thus, $X_\perp X_\perp^\top VX_\perp X_\perp^\top=0$ so $X_\perp^\top VX_\perp=0$. Conversely, if $X_\perp^\top VX_\perp=0$, then $V$ is the image by $d_X\pi$ of a horizontal vector so $V\in T_\Sigma\Sym^+(n,k)$. Hence $T_\Sigma\Sym^+(n,k)=\{V\in\Sym(n)|X_\perp^\top VX_\perp=0\}$.
    
    \item[3.] The quotient metric is defined by $g^{\mathrm{BW}(n,k)}_\Sigma(V,V)=\tr(V^\#_X(V^\#_X)^\top)$ so we only need to compute $V^\#(V^\#)^\top$ for any $X$, for example $X=UD^{1/2}$, and its trace.
    \begin{align*}
        V^\#&=UD^{1/2}D^{-1}\mc{S}_D(D^{1/2}U^\top VUD^{1/2})\\
        &\quad+(I_n-UD^{1/2}D^{-1}D^{1/2}U^\top)VUD^{1/2}D^{-1}\\
        &=U\mc{S}_D(U^\top VU)D^{1/2}+(I_n-UU^\top)VUD^{-1/2}\\
        &=SUD^{1/2}+(I_n-UU^\top)VUD^{-1/2},\\
        V^\#(V^\#)^\top&=SUDU^\top S+SV(I_n-UU^\top)+(I_n-UU^\top)VS\\
        &\quad+(I_n-UU^\top)V\Sigma^- V(I_n-UU^\top).
    \end{align*}
    Hence $\tr(V^\#(V^\#)^\top)=\tr(S\Sigma S+V\Sigma^- V(I_n-UU^\top))$.
\end{enumerate}
\end{proof}

\subsection{Theorem \ref{thm:bw_geodesics_sym^+(n,k)}}\label{subsec:proof:thm:bw_geodesics_sym^+(n,k)}

(\chixthmbwgeodsymnktitle) Let $\Sigma,\Lambda\in\Sym^+(n,k)$ and $X,Y\in\R^{n\times k}_*$ such that $XX^\top=\Sigma$ and $YY^\top=\Lambda$. Let $U\in\St(n,k)$ and $D\in\Diag^+(k)$ such that $\Sigma=UDU^\top$.
\begin{enumerate}
    \item (Exponential map) For all $V\in T_\Sigma\Sym^+(n,k)$, the geodesic from $\Sigma$ with initial speed $V$ is $\gamma_{(\Sigma,V)}:t\in I_{\Sigma,V}\lmto\Sigma+tV+t^2W_{\Sigma,V}$, where $W_{\Sigma,V}=S_{\Sigma,V}\Sigma S_{\Sigma,V}+S_{\Sigma,V}V(I_n-UU^\top)+(I_n-UU^\top)VS_{\Sigma,V}+(I_n-UU^\top)V\Sigma^+ V(I_n-UU^\top)$ and $S_{\Sigma,V}=U\mc{S}_D(U^\top VU)U^\top$.
    
    \item (Definition interval) Let $F^0_{X,V}=\mc{S}_{X^\top X}((X^\top X)^{-1/2}X^\top VX(X^\top X)^{-1/2})$ and $M^0_{X,V}=(X^\top X)^{-3/2}X^\top V(I_n-X(X^\top X)^{-1}X^\top)VX(X^\top X)^{-3/2}\in\Sym(n)$. Let $\mc{E}_{\Sigma,V}=\{\lambda\in\sp(F^0_{X,V})|\,\ker(\lambda I_k-F^0_{X,V})\cap\ker(M^0_{X,V})\ne\{0\}\}\subseteq\sp(S_{\Sigma,V})$. If $\mc{E}_{\Sigma,V}$ is non-empty, then let $\lambda_+=\max\mc{E}_{\Sigma,V}$ and $\lambda_-=\min\mc{E}_{\Sigma,V}$. The definition interval of the geodesic $\gamma_{(\Sigma,V)}$ is the interval $I_{\Sigma,V}$ defined by:
    \begin{enumerate}[label=$\cdot$]
        \item $I_{\Sigma,V}=(-\frac{1}{\lambda_+},-\frac{1}{\lambda_-})$ if $\lambda_-<0<\lambda_+$,
        \item $I_{\Sigma,V}=(-\infty,-\frac{1}{\lambda_-})$ if $\lambda_-<0$ and $\lambda_+\ls 0$,
        \item $I_{\Sigma,V}=(-\frac{1}{\lambda_+},+\infty)$ if $\lambda_-\gs 0$ and $\lambda_+>0$,
        \item $I_{\Sigma,V}=\R$ if $\mc{E}_{\Sigma,V}$ is empty.
    \end{enumerate}
    
    Applying this to $X=UD^{1/2}$ without loss of generality, $F^0_{X,V}=\mc{S}_D(U^\top VU)$ and $M^0_{X,V}=D^{-1}U^\top V(I_n-UU^\top)VUD^{-1}$ which is a bit more tractable to compute $\mc{E}_{\Sigma,V}$.
    
    \item (Cut time) Let $\lambda_\maxi=\max\sp(F^0_{X,V})$ and $\lambda_\mini=\min\sp(F^0_{X,V})$. Note that if $\mc{E}_{\Sigma,V}\ne\emptyset$, then we have $(\lambda_-,\lambda_+)\subseteq(\lambda_\mini,\lambda_\maxi)$. The cut time is $t_{cut}(\Sigma,V)=-\frac{1}{\lambda_\mini}$ if $\lambda_\mini<0$ or $+\infty$ otherwise. Symmetrically, we have $t_{cut}(\Sigma,-V)=\frac{1}{\lambda_\maxi}$ if $\lambda_\maxi>0$ or $+\infty$ otherwise.
    
    \item (Preimages) We define the indexing set $\mc{I}^{\mc{P}re}_{X,Y}$ by:
    \begin{align*}
        \mc{I}^{\mc{P}re}_{X,Y}=\{R\in\Orth(n)|&H:=X^\top YR^\top\in\Sym(n)\mathrm{~and~}\\
        &\forall\mu<0,\ker(\mu I_k-(X^\top X)^{-1/2}H(X^\top X)^{-1/2})\\
        &\cap\ker(\mu^2 I_k-(X^\top X)^{-1/2}RY^\top YR^\top (X^\top X)^{-1/2})=\{0\}\}.
    \end{align*}

    For $R\in\mc{I}^{\mc{P}re}_{X,Y}$, we denote $H=H_{X,Y,R}=X^\top YR^\top$ so that $X^\top Y=HR$. Then the map $R\in\mc{I}^{\mc{P}re}_{X,Y}\lmto V=2\,\sym(XRY^\top)-2\Sigma\in\mc{P}re_\Sigma(\Lambda)$ is a bijection whose inverse is $V\in\mc{P}re_\Sigma(\Lambda)\lmto R=(Y^\top Y)^{-1}Y^\top(X+V^\#_X)\in\mc{I}^{\mc{P}re}_{X,Y}$.
    
    The geodesic joining $\Sigma$ to $\Lambda$ parametrized by $R\in\mc{I}^{\mc{P}re}_{X,Y}$ writes:
    \begin{equation*}
        \forall t\in[0,1],\gamma_{\Sigma\to\Lambda}^R(t)=(1-t)^2\Sigma+t^2\Lambda+2t(1-t)\sym(XRY^\top).
    \end{equation*}
    
    \item (Logarithms) Let $\mc{I}^{\mc{L}og}_{X,Y}=\{R\in\Orth(n)|H_{X,Y,R}=X^\top YR^\top\in\Cov(n)\}=\{R\in\Orth(n)|H_{X,Y,R}=(X^\top\Lambda X)^{1/2}\}=\{R\in\Orth(n)|\,X^\top Y=(X^\top\Lambda X)^{1/2}R\}\subseteq\mc{I}^{\mc{P}re}_{X,Y}$.
    
    Then, the map $R\in\mc{I}^{\mc{L}og}_{X,Y}\lmto V=2\,\sym(XRY^\top)-2\Sigma\in\mc{L}og_\Sigma(\Lambda)$ is a bijection whose inverse is $V\in\mc{L}og_\Sigma(\Lambda)\lmto R=(Y^\top Y)^{-1}Y^\top(X+V^\#_X)\in\mc{I}^{\mc{L}og}_{X,Y}$.
    
    \item (Logarithm map) Let $r=\rk(\Sigma\Lambda)=\rk(X^\top Y)=\rk(H)$.
    \begin{enumerate}
        \item If $r=k$, then there exists a unique logarithm of $\Lambda$ from $\Sigma$. In this case, the minimizing geodesic joining $\Sigma$ to $\Lambda$ writes:
        \small
        \begin{equation}
            \gamma_{\Sigma\to\Lambda}(t)=(1-t)^2\Sigma+t^2\Lambda+2t(1-t)\sym(\Sigma^{1/2}((\Sigma^{1/2}\Lambda\Sigma^{1/2})^{1/2})^-\Sigma^{1/2}\Lambda).
        \end{equation}
        \normalsize
        \item If $r=k-1$, then there exist exactly two logarithms of $\Lambda$ from $\Sigma$.
        
        \item If $r<k-1$, then there is an infinity of logarithms of $\Lambda$ from $\Sigma$.
    \end{enumerate}
    Therefore, the logarithm map is defined on $\mc{U}_\Sigma=\{\Lambda\in\Sym^+(n,k)|\rk(\Sigma\Lambda)=k\}$ and it writes $\Log_\Sigma:\Lambda\in\mc{U}_\Sigma\lmto 2\,\sym(\Sigma^{1/2}((\Sigma^{1/2}\Lambda\Sigma^{1/2})^{1/2})^-\Sigma^{1/2}\Lambda)-2\Sigma\in T_\Sigma\Sym^+(n,k)$.
\end{enumerate}

\begin{proof}[Proof of Theorem \ref{thm:bw_geodesics_sym^+(n,k)}]
We prove statement 3 in the end because it requires statements 4 and 5.
\begin{enumerate}
    \item (Exponential map) By Theorem \ref{thm:geodesics_quotient} \cite{ONeill66}, the exponential map is simply $\Exp_\Sigma(tV)=\pi(\Exp_X(tV^\#_X))=(X+tV^\#_X)(X+tV^\#_X)^\top=\Sigma+tV+V^\#_X(V^\#_X)^\top$. The matrix $W=V^\#_X(V^\#_X)^\top$ was already computed in the proof of Theorem \ref{thm:horlift_metric}.
    
    \item (Definition domain) As in $\Sym^+(n)$, let us first determine $J_{\Sigma,V}=\{t\in\R|\rk(\Sigma+tV+t^2W_{\Sigma,V})=k\}$. According to \cite[Proposition 3.2]{Massart20} applied to $tV^\#_X=t[X(X^\top X)^{-1}F_{X,V}+X_\perp K_{X,V}]$ with $F=F_{X,V}=\mc{S}_{X^\top X}(X^\top VX)$ and $K=K_{X,V}=X_\perp^\top VX(X^\top X)^{-1}$, we have $t\in J_{\Sigma,V}$ if and only if $\ker(I_k+t(X^\top X)^{-1}F)\cap\ker(K)=\{0\}$. Let $\mc{E}_{\Sigma,V}=\{\lambda\in\sp((X^\top X)^{-1}F)|\,\ker(\lambda I_k-(X^\top X)^{-1}F)\cap\ker(K)\ne\{0\}\}$. It is clear that $J_{\Sigma,V}=\R\backslash{\{-\frac{1}{\lambda}|\lambda\in\mc{E}_{\Sigma,V}\}}$. Then $I_{\Sigma,V}$ is the connected component of $0$ in $J_{\Sigma,V}$. Its computation is analogous to the one in the proof of Theorem \ref{thm:bw_geodesics_sym^+(n)}.
    
    To get a condition in $\mc{E}_{\Sigma,V}$ that directly depends on $V$, note that the condition rewrites $\ker(\lambda I_k-(X^\top X)^{-1/2}F(X^\top X)^{-1/2})\cap\ker(K(X^\top X)^{-1/2})\ne\{0\}$ and $\ker(K(X^\top X)^{-1/2})=\ker((X^\top X)^{-1/2}K^\top K(X^\top X)^{-1/2})$ with $K^\top K=(X^\top X)^{-1}X^\top V(I_n-X(X^\top X)^{-1}X^\top)VX(X^\top X)^{-1}$. Note that $\mc{E}_{\Sigma,V}$ is independent from the choice of $X$ because $F_{XR,V}=R^\top F_{X,V}R$ and $K_{XR,V}=R^\top K_{X,V}R$ for all $R\in\Orth(n)$ so the condition does not depend on $X$.

    \item[4.] (Preimages) In \cite[Propositions 4.4 \& 4.5]{Massart20}, the solutions of the equation $\Exp_\Sigma(V)=\Lambda$ are not computed in the definition domain $\mc{D}_\Sigma$ of the exponential map but in the wider set $\{V\in T_\Sigma\Sym^+(n,k)|\Sigma+V+W_{\Sigma,V}\in\Sym^+(n,k)\}$. Thus, the geodesic $\gamma_{(\Sigma,V)}$ may leave the manifold $\Sym^+(n,k)$ before reaching $\Lambda$. Therefore, we complete their work with the additional condition $1\in I_{\Sigma,V}$ to characterize the preimages of $\Lambda$ from $\Sigma$.
    
    From \cite{Massart20}, we know that preimages $V$ necessarily satisfy $V^\#_X=YR^\top-X$ with $X^\top Y=HR$, $H\in\Sym(n)$, $R\in\Orth(n)$. Thus, $V=XRY^\top+YR^\top X^\top-2XX^\top$ so $X^\top VX=HX^\top X+X^\top XH-2(X^\top X)^2$ so $F=\mc{S}_{X^\top X}(X^\top VX)=H-X^\top X$. Moreover, $K=X_\perp^\top VX(X^\top X)^{-1}=X_\perp^\top YR^\top$. Denoting $A=(X^\top X)^{-1/2}H(X^\top X)^{-1/2}$ and $B=(X^\top X)^{-1/2}RY^\top YR^\top(X^\top X)^{-1/2}$, we have $(X^\top X)^{-1/2}F(X^\top X)^{-1/2}=A-I_k$ and $(X^\top X)^{-1/2}K^\top K(X^\top X)^{-1/2}=(X^\top X)^{-1/2}RY^\top X_\perp X_\perp^\top YR^\top(X^\top X)^{-1/2}=B-A^2$.
    
    We can now compute $\mc{E}_{\Sigma,V}$ and $I_{\Sigma,V}$. For all $\lambda\in\R$, for all $Z\in\R^n$:
    \begin{align*}
        &Z\in\ker(\lambda I_k-S^0_{X,V})\cap\ker(M^0_{X,V})\\
        &\cns Z\in\ker((\lambda+1)I_k-A)\cap\ker(B-A^2)\\
        &\cns AZ=(\lambda+1)Z\mathrm{~and~}BZ=A^2Z\\
        &\cns AZ=(\lambda+1)Z\mathrm{~and~}BZ=(\lambda+1)^2Z\\
        &\cns Z\in\ker((\lambda+1)I_k-A)\cap\ker((\lambda+1)^2I_k-B).
    \end{align*}
    Therefore:
    \begin{align*}
        \mc{E}_{\Sigma,V}&=\{\lambda\in\sp(A-I_k)|\ker((\lambda+1)I_k-A)\cap\ker((\lambda+1)^2I_k-B)\ne\{0\}\}\\
        &=\{\mu-1\in\sp(A-I_k)|\ker(\mu I_k-A)\cap\ker(\mu^2I_k-B)\ne\{0\}\}.
    \end{align*}
    Thus, denoting $\lambda_-=\min\mc{E}_{\Sigma,V},$ the condition $1\in I_{\Sigma,V}$ rewrites:
    \begin{align*}
        1\in I_{\Sigma,V}&\cns\lambda_-\gs 0\mathrm{~or~}-\frac{1}{\lambda_-}\gs 1\cns\lambda_-\gs-1\\
        &\cns\forall\mu\in\sp(A),\ker(\mu I_k-A)\cap\ker(\mu^2I_k-B)\ne\{0\}\nec\mu\gs0\\
        &\cns\forall\mu<0,\ker(\mu I_k-A)\cap\ker(\mu^2I_k-B)=\{0\}.
    \end{align*}
    
    To conclude, with the notations of statement 4, $V=2\,\sym(XRY^\top)-2\Sigma\in\mc{P}re_\Sigma(\Lambda)$ if and only if $R\in\mc{I}^{\mc{P}re}_{X,Y}$.
    
    \item[5.] (Logarithms) In \cite{Massart20}, it is stated that the shortest vectors $V=d\pi(V^\#_X)$ with $V^\#_X=YR^\top-X$ are those for which $(H,R)$ is a polar decomposition of $X^\top Y$, i.e. $H\gs 0$. Therefore, we necessarily have $H=(X^\top\Lambda X)^{1/2}$. It is well known that they even satisfy $\|V\|=\|V^\#\|=d^{\mathrm{BW}}(\Sigma,\Lambda)$ (see Definition \ref{def:BW_distance}). 
    Moreover, if $H\gs 0$, the condition $1\in I_{\Sigma,V}$ is automatically satisfied, as stated in \cite[Corollary 3.3 (5)]{Massart20}. So with the notations of statement 5, the logarithms are indexed by $\mc{I}^{\mc{L}og}_{X,Y}$.
    
    \item[6.] (Logarithm map) We have $\rk(\Sigma\Lambda)\ls\rk(X^\top Y)$ since $\Sigma\Lambda=X(X^\top Y)Y^\top$. We also have $X^\top Y=(X^\top X)^{-1}X^\top(\Sigma\Lambda)Y(Y^\top Y)^{-1}$ so $\rk(X^\top Y)\ls\rk(\Sigma\Lambda)$. Finally, $\rk(\Sigma\Lambda)=\rk(X^\top Y)=\rk(H)$. We denote it $r=\rk(\Sigma\Lambda)$.
    
    \begin{enumerate}
        \item As stated in \cite{Massart20}, if $r=k$, then there exists a unique logarithm of $\Sigma$ from $\Lambda$. Moreover, we can compute an explicit expression. Indeed, $R=H^{-1}X^\top Y$ so $XRY^\top=XH^{-1}X^\top\Lambda=X(X^\top\Lambda X)^{-1/2}X^\top\Lambda$. Since the choice of $X$ is free, let us take $X=UD^{1/2}$ where $\Sigma=UDU^\top$ with $U\in\St(n,k)$ and $D\in\Diag^+(k)$. Therefore:
        \begin{align*}
            XRY^\top&=UD^{1/2}(D^{1/2}U^\top\Lambda UD^{1/2})^{-1/2}D^{1/2}U^\top\Lambda\\
            &=UD^{1/2}U^\top U((D^{1/2}U^\top\Lambda UD^{1/2})^{1/2})^{-1}U^\top UD^{1/2}U^\top\Lambda\\
            &=\Sigma^{1/2}(U(D^{1/2}U^\top\Lambda UD^{1/2})^{1/2}U^\top)^-\Sigma^{1/2}\Lambda\\
            &=\Sigma^{1/2}((UD^{1/2}U^\top\Lambda UD^{1/2}U^\top)^{1/2})^-\Sigma^{1/2}\Lambda\\
            &=\Sigma^{1/2}((\Sigma^{1/2}\Lambda\Sigma^{1/2})^{1/2})^-\Sigma^{1/2}\Lambda.
        \end{align*}
        So the unique minimizing geodesic joining $\Sigma$ to $\Lambda$ writes:
        \begin{equation*}
            \gamma_{(\Sigma,\Lambda)}(t)=(1-t)^2\Sigma+t^2\Lambda+2t(1-t)\sym(\Sigma^{1/2}((\Sigma^{1/2}\Lambda\Sigma^{1/2})^{1/2})^-\Sigma^{1/2}\Lambda).
        \end{equation*}
        
        \item If $r=k-1$, without loss of generality, let us assume that $X^\top Y\in\Diag(k)$, $X^\top Y=\Diag(d_1,...,d_{k-1},0)$. Then $H\gs 0$ and $H^2=(X^\top Y)^2$ imposes that $H=\Diag(|d_1|,...,|d_{k-1}|,0)$. Therefore, there are only two matrices $R_\pm\in\Orth(n)$ defined by $R_\pm=\Diag(\mathrm{sgn}(d_1),...,\mathrm{sgn}(d_{k-1}),\pm 1)$ such that $X^\top Y=HR_\pm$. Thus there are exactly two logarithms of $\Lambda$ from $\Sigma$.
        
        \item If $r<k-1$, similarly we can assume without loss of generality that $X^\top Y=\Diag(d_1,...,d_r,0,...,0)$. Then $H=\Diag(|d_1|,...,|d_r|,0,...,0)$ and $R=\Diag(\varepsilon,R_0)$ is a block-diagonal matrix with $R_0\in\Orth(k-r)$ and $\varepsilon=\Diag(\mathrm{sgn}(d_1),...,\mathrm{sgn}(d_r))\in\Diag(r)$. Thus there is an infinity of logarithms of $\Lambda$ from $\Sigma$.
    \end{enumerate}
    Thus the logarithm map is defined on $\mc{U}_\Sigma=\{\Lambda\in\Sym^+(n,k)|\,\rk(\Sigma\Lambda)=k\}$, as stated in \cite{Massart20}.
    
    \item[3.] (Cut time) Let $t\in I_{\Sigma,V}\cap\R_+$. Let $\Lambda=\gamma_{(\Sigma,V)}(t)$, $Y\in\R^{n\times k}_*$ such that $YY^\top=\Lambda$, $(H,R)\in\Sym(k)\times\Orth(k)$ such that $X^\top Y=HR$ and $V^\#_X=YR^\top-X$. Then, $X^\top X+tX^\top V^\#_X=X^\top YR^\top=H$. Besides, $X^\top X+tX^\top V^\#=(X^\top X)^{1/2}(I_n+tS^0_{X,V})(X^\top X)^{1/2}$. Let $\lambda_\mini=\min\sp(S^0_{X,V})=\min\sp(S_{\Sigma,V})$. Therefore, $\gamma_{(\Sigma,V)}$ is minimizing on $[0,t]$ if and only if $H\gs 0$ if and only if $I_n+tS^0_{X,V}\gs 0$ if and only if $1+t\lambda_\mini\gs 0$. If $\lambda_\mini\gs 0$, the condition is empty so $t_{cut}(\Sigma,V)=+\infty$. If $\lambda_\mini<0$, the condition writes $t\ls-\frac{1}{\lambda_\mini}$ so $t_{cut}(\Sigma,V)=-\frac{1}{\lambda_\mini}$.
\end{enumerate}
\end{proof}

\subsection{Theorem \ref{thm:minimizing_geodesics}}\label{subsec:proof:thm:minimizing_geodesics}

(\chixthmmingeodtitle) \chixthmmingeod

\begin{proof}[Proof of Theorem \ref{thm:minimizing_geodesics}]
~\\
(Necessity) Let $\gamma:[0,1]\lto\Cov(n)$ be a minimizing geodesic segment from $\Sigma=\gamma(0)$ to $\Lambda=\gamma(1)$. Let $p=\max_{t\in[0,1]}\rk(\gamma(t))$. By Lemma \ref{lemma:rank}, 
$\gamma$ is of constant rank $p\gs\max(k,l)$ on $(0,1)$. In other words, $\gamma_{|(0,1)}:(0,1)\lto\Sym^+(n,p)$ is a minimizing geodesic of $\Sym^+(n,p)$. Let $c_0:(0,1)\lto\R^{n\times p}_*$ be a horizontal lift of $\gamma_{|(0,1)}$. Necessarily, $c_0(t)=(1-t)X_0+tY_0$ with $X_0,Y_0\in\R^{n\times p}$ with $X_0X_0^\top=\Sigma$ and $Y_0Y_0^\top=\Lambda$ since $(X_0,\Sigma)=\lim_{t\to 0}(c_0(t),\gamma(t))$ and $(Y_0,\Lambda)=\lim_{t\to 1}(c_0(t),\gamma(t))$. Let us show that $X_0^\top Y_0\in\Cov(p)$. 

For all $[a,b]\subset(0,1)$, the tangent vectors $V^\#_{c_0(a)}=(b-a)(Y_0-X_0)\in T_{c_0(a)}\R^{n\times p}_*=\R^{n\times p}$ and $V=c_0(a)(V^\#_{c_0(a)})^\top+V^\#_{c_0(a)}c_0(a)^\top\in\mc{L}og_{\gamma(a)}(\gamma(b))\subset T_{\gamma(a)}\Sym^+(n,p)$ uniquely determine a pair of matrices $R_{a,b}\in\mc{I}^{\mc{L}og}_{c_0(a),c_0(b)}$ and $H_{a,b}=H_{c_0(a),c_0(b),R_{a,b}}=c_0(a)^\top c_0(b)\in\Cov(p)$. We compute $H_{a,b}$:
\begin{align*}
    H_{a,b}&=c_0(a)^\top c_0(b)=[(1-a)X_0+aY_0]^\top[(1-b)X_0+bY_0]\\
    &=(1-a)(1-b)X_0^\top X_0+abY_0^\top Y_0+(1-a)bX_0^\top Y_0+a(1-b)Y_0^\top X_0.
\end{align*}
Therefore, $\lim_{\substack{a\to 0\\b\to 1}}H_{a,b}=X_0^\top Y_0$ so $X_0^\top Y_0\in\Cov(p)$.

Since $[X_0~0][X_0~0]^\top=\Sigma$ and $[Y_0~0][Y_0~0]^\top=\Lambda$, there exist $P,Q\in\Orth(n)$ such that $X=[X_0~0]P$ and $Y=[Y_0~0]Q$. Thus the curve $c:t\in[0,1]\lmto(1-t)X+tYR^\top\in\Mat(n)$ with $R=P^\top Q\in\Orth(n)$ satisfies $\gamma(t)=c(t)c(t)^\top$ for all $t\in[0,1]$. Indeed, it is equal to $c_0(t)c_0(t)^\top$ on $(0,1)$ and the equality is clear for $t\in\{0,1\}$. Moreover, $H_{X,Y,R}=X^\top YR^\top=U^\top\Diag(X_0^\top Y_0,0)U\in\Cov(n)$ so $H_{X,Y,R}=((X^\top Y)(X^\top Y)^\top)^{1/2}=(X^\top\Lambda X)^{1/2}$. Finally, $\gamma(t)=c(t)c(t)^\top=(1-t)^2\Sigma+t^2\Lambda+2t(1-t)\,\sym(XRY^\top)$.\\

(Sufficiency) Let $\gamma_{(\Sigma,\Lambda)}^R(t)=(1-t)^2\Sigma+t^2\Lambda+2t(1-t)\sym(XRY^\top)$ with $H=H_{X,Y,R}=X^\top YR^\top\in\Cov(n)$ and let us prove that it is a minimizing geodesic segment. We define $W=YR^\top-X$ and $c(t)=X+tW=(1-t)X+tYR^\top$ for $t\in[0,1]$. The curve $c$ is a geodesic of $\Mat(n)$ such that $c(t)c(t)^\top=\gamma_{(\Sigma,\Lambda)}^R(t)$ for all $t\in[0,1]$. Moreover, $L(c)=\|W\|=\tr(XX^\top+YY^\top-2X^\top YR^\top)^{1/2}=\tr(\Sigma+\Lambda-2H)^{1/2}$. Let $Q\in\Orth(n)$ such that $X=\Sigma^{1/2}Q$. Since $H\gs 0$, $H=(X^\top\Lambda X)^{1/2}=Q^\top(\Sigma^{1/2}\Lambda\Sigma^{1/2})^{1/2}Q$. Therefore, $L(c)=\|YR^\top-X\|=d^{\mathrm{BW}}(\Sigma,\Lambda)$. In other words, $c:[0,1]\lto\Cov(n)$ is a minimizing curve between two registered points $X$ and $YR^\top$ so by Lemma \ref{lemma:length_quotient}, 
its projection $\gamma:[0,1]\lto\Cov(n)$ is a minimizing curve and $L(\gamma)=L(c)=d^{\mathrm{BW}}(\Sigma,\Lambda)$.

By Lemma \ref{lemma:rank} 
again, $\gamma$ has constant rank $p\gs\max(k,l)$ on $(0,1)$ so $\gamma_{|(0,1)}:(0,1)\lto\Sym^+(n,p)$ is a minimizing curve of $\Sym^+(n,p)$. Since $c_{|(0,1)}$ has constant speed, so does $\gamma_{|(0,1)}$. By continuity of the length, $\gamma$ has constant speed on $[0,1]$ so $\gamma:[0,1]\lto\Cov(n)$ is a minimizing geodesic segment.
\end{proof}

\subsection{Lemma \ref{lemma:bures-wasserstein_elem_algebra}}\label{subsec:proof:lemma:bures-wasserstein_elem_algebra}

(\chixlemmaelemalgtitle) \chixlemmaelemalg

\begin{proof}[Proof of Lemma \ref{lemma:bures-wasserstein_elem_algebra}]
\begin{enumerate}
    \item Let $X_0\in\R^{n\times k}_*$ and $Y_0\in\R^{n\times l}_*$ such that $X_0X_0^\top=\Sigma$ and $Y_0Y_0^\top=\Lambda$. Thus there exist $P,Q\in\Orth(n)$ such that $X=[X_0~0]P$ and $Y=[Y_0~0]Q$. Since $\Sigma\Lambda=X(X^\top Y)Y^\top$, we have $r\ls\rk(X^\top Y)=\rk(X_0^\top Y_0)$. Since $X_0^\top Y_0=(X_0^\top X_0)^{-1}X_0^\top\Sigma\Lambda Y_0(Y_0^\top Y_0)^{-1}$, we have $\rk(X_0^\top Y_0)\ls\rk(\Sigma\Lambda)$. Finally, $r=\rk(X_0^\top Y_0)=\rk(X^\top Y)$.
    
    \item Let $f,g:\R^n\lto\R^n$ be linear endomorphisms respectively represented by $\Sigma$ and $\Lambda$ is the canonical basis. From the rank-nullity theorem applied to the linear map $f_{|\im(g)}:\im(g)\lto\R^n$, i.e. the restriction of $f$ to $\im(g)$, and since $\im(f_{|\im(g)})=\im(f\circ g)$ and $\ker(f_{|\im(g)})\subseteq\ker(f)$, we have $\rk(g)=\rk(f_{|\im(g)})+\dim\ker(f_{|\im(g)})\ls\rk(f\circ g)+\dim(\ker f)=\rk(f\circ g)+n-\rk(f)$. This writes $l-r\ls n-k$.
\end{enumerate}
\end{proof}

\subsection{Theorem \ref{thm:number_bw_minimizing_geodesics}}\label{subsec:proof:thm:number_bw_minimizing_geodesics}

(\chixthmnbbwgeodtitle) Let $\Sigma,\Lambda\in\Cov(n)$ with $\rk(\Sigma)=k$ and $\rk(\Lambda)=l$. We assume that $k\gs l$ without loss of generality. We denote $r=\rk(\Sigma\Lambda)$. We have $l-r\ls n-k$.
\begin{enumerate}
    \itemsep0em
    \item There exists a bijection between the set of minimizing geodesics from $\Sigma$ to $\Lambda$ and the closed unit ball of $\R^{(k-r)\times(l-r)}$ for the spectral norm $\bar{\mc{B}}_{\mathrm{S}}(0,1)=\{R_0\in\R^{(k-r)\times(l-r)}|\,\|R_0\|_{\mathrm{S}}\ls 1\}=\{R_0\in\R^{(k-r)\times(l-r)}|\,0\ls R_0^\top R_0\ls I_{l-r}\}$.
    \item The minimizing geodesic is unique if and only if $r=l$. This includes the cases $k=n$.
    \item There is an infinite number of minimizing geodesics if and only if $r<l$.
    \item The minimizing geodesics corresponding to the choices $R_0\in\St(k-r,l-r)$ (including the empty matrix if $r=l$) have rank exactly $k$ on $[0,1)$ (on $[0,1]$ if $l=k$). Note that $\St(k-r,l-r)$ is included in the unit sphere $\mc{S}_{\mathrm{S}}(0,1)=\{R_0\in\R^{(k-r)\times(l-r)}|\,\|R_0\|_{\mathrm{S}}=1\}$.
    \item The minimizing geodesic corresponding to the choice $R_0=0$ (or the empty matrix if $r=l$) writes for all $t\in[0,1]$:
    \small
    \begin{equation*}
        \gamma^0_{\Sigma\to\Lambda}(t)=(1-t)^2\Sigma+t^2\Lambda+2t(1-t)\,\sym(\Sigma^{1/2}((\Sigma^{1/2}\Lambda\Sigma^{1/2})^{1/2})^-\Sigma^{1/2}\Lambda).
    \end{equation*}
    \normalsize
    If $r=l$, it has rank exactly $k$ on $[0,1)$.
\end{enumerate}
The number of minimizing geodesic segments in $\Sym^+(n,k)$ and in $\Cov(n)$ is summarized in Table \ref{tab:minimizing_geodesics} 
with $n\gs k\gs l\gs r$.
\begin{center}
\begin{tabular}{|c|c|c|c|c|}
    \hline
    \multirow{2}{*}{$\Sigma\in$} & \multirow{2}{*}{$\Lambda\in$} & \multirow{2}{*}{$r=\rk(\Sigma\Lambda)$} & \multicolumn{2}{c|}{Number of minimizing geodesics}\\
    \cline{4-5}
     &  &  & $~\mathrm{in~}\Sym^+(n,k)~$ & $\mathrm{in~}\Cov(n)$\\
    \hline
    $\Sym^+(n)$ & $\Sym^+(n)$ & $n$ & $1$ & $1$\\
    \hline
    $\Sym^+(n)$ & $\Sym^+(n,k)$ & $k$ & $1$ & $1$\\
    \hline
    \multirow{3}{*}{$\Sym^+(n,k)$} & \multirow{3}{*}{$\Sym^+(n,k)$} & $k$ & $1$ & $1$\\
    \cline{3-5}
    & & $k-1$ & $2$ & $\infty$\\
    \cline{3-5}
    & & $<k-1$ & $\infty$ & $\infty$\\
    \hline
    \multirow{2}{*}{$\Sym^+(n,k)$} & \multirow{2}{*}{$\Sym^+(n,l)$} & $l$ & $1$ & $1$\\
    \cline{3-5}
    & & $<l$ & $\infty$ & $\infty$\\
    \hline
\end{tabular}\\
\vspace{3mm}
\hbox{Table \ref{tab:minimizing_geodesics}:
Number of Bures-Wasserstein minimizing geodesics ($n\gs k\gs l\gs r$)}
\end{center}

\begin{proof}[Proof of Theorem \ref{thm:number_bw_minimizing_geodesics}]
\begin{enumerate}
    \item Without loss of generality, let us choose $X,Y\in\Mat(n)$ with an convenient form. We choose $X_0\in\R^{n\times k}_*$ and $Y_0\in\R^{n\times l}_*$ as in the proof of the previous lemma. Given a singular value decomposition of $X_0^\top Y_0=U_kDV_l^\top$ with $U_k\in\Orth(k)$, $V_l\in\Orth(l)$ and $D=\Diag(D_r,0)$ with $D_r\in\Diag^+(r)$, we define $P=\Diag(U_k,I_{n-k})\in\Orth(n)$ and $Q=\Diag(V_l,I_{n-l})\in\Orth(n)$. We choose $X=[X_0~0]P=[X_k~0]$ with $X_k=X_0U_k\in\R^{n\times k}_*$ and $Y=[Y_0~0]Q=[Y_l~0]$ with $Y_l=Y_0V_l\in\R^{n\times l}_*$. Therefore we have $XX^\top=\Sigma$, $YY^\top=\Lambda$, $X=[X_k~0]$, $Y=[Y_l~0]$ and $X^\top Y=\Diag(D_r,0)$. We denote $X=[X_r~X_{k-r}~0]$ and $Y=[Y_r~Y_{l-r}~0]$ with $X_r,Y_r\in\R^{n\times r}_*$, $X_{k-r}\in\R^{n\times(k-r)}$ and $Y_{l-r}\in\R^{n\times(l-r)}$.
    
    Necessarily, $H=(X^\top YY^\top X)^{1/2}=\Diag(D_r,0)$. The possible $R\in\Orth(n)$ such that $HR=X^\top Y$ are $R=\Diag(I_r,R_{n-r})$ with $R_{n-r}=[R_{l-r}~R_{n-l}]\in\Orth(n-r)$ where $R_{l-r}\in\St(n-r,l-r)$ and $R_{n-l}\in\St(n-r,n-l)$. We denote $R_{l-r}=\begin{pmatrix}R_0\\R_1\end{pmatrix}$ with $R_0\in\R^{(k-r)\times(l-r)}$ and $R_1\in\R^{(n-k)\times(l-r)}$ with $R_0^\top R_0+R_1^\top R_1=I_{l-r}$. Note that both $R_0$ and $R_1$ have more rows than columns. 
    
    A simple calculus gives $XRY^\top=X_rY_r^\top+X_{k-r}R_0Y_{l-r}^\top$. Given $R,R'$ satisfying $HR=HR'=X^\top Y$, we have $XRY^\top=XR'Y^\top$ if and only if $R_0=R'_0$ since $X_{k-r}^\top X_{k-r}$ and $Y_{l-r}^\top Y_{l-r}$ are invertible. We even have $\sym(XRY^\top)=\sym(XR'Y^\top)$ if and only if $R_0=R'_0$. Indeed if $\sym(XRY^\top)=\sym(XR'Y^\top)$, then $X_{k-r}(R_0-R_0')Y_{l-r}^\top=Y_{l-r}(R_0'-R_0)^\top X_{k-r}^\top$. Since $X_{k-r}^\top Y_{l-r}=0$, it suffices to multiply on the left by $X_{k-r}^\top$ and on the right by $Y_{l-r}$ to conclude that $R_0=R'_0$. Thus there is a bijection between minimizing geodesic segments and submatrices $R_0\in\R^{(k-r)\times(l-r)}$ of $R_{l-r}=\begin{pmatrix}R_0\\R_1\end{pmatrix}\in\St(n-r,l-r)$. Since $R_1$ has more rows than columns, any $R_0\in\R^{(k-r)\times(l-r)}$ such that $R_0^\top R_0\ls I_{l-r}$ can be completed by an appropriate $R_1=\begin{pmatrix}(I_{l-r}-R_0^\top R_0)^{1/2}\\\mathbf{0}_{n-k-(l-r),l-r}\end{pmatrix}$.
    
    Therefore, the minimizing geodesic segments are in bijection with the matrices $R_0\in\R^{(k-r)\times(l-r)}$ such that $R_0^\top R_0\ls I_{l-r}$, that is the closed unit ball for the spectral norm $\bar{\mc{B}}_{\mathrm{S}}(0,1)$.
    
    \item When $r=l$, the component $Y_{l-r}$ of $Y$ is the empty matrix. In other words, the dependence of the minimizing geodesic on $R_0$ vanishes so the minimizing geodesic is unique. In particular when $k=n$, $r=\rk(\Sigma\Lambda)=\rk(\Lambda)=l$.
    
    \item On the contrary, when $r<l$ (thus $n>k\gs l$), there is an infinite number of convenient $R_0$'s. For example, $R_0=\begin{pmatrix}\cos\theta & \mathbf{0}_{1,l-r-1}\\\mathbf{0}_{k-r-1,1} & \mathbf{0}_{k-r-1,l-r-1}\end{pmatrix}$ and $R_1=\begin{pmatrix}\sin\theta & \mathbf{0}_{1,l-r-1}\\\mathbf{0}_{l-r-1,1} & I_{l-r-1}\\\mathbf{0}_{n-k-(l-r),1} & \mathbf{0}_{n-k-(l-r),l-r-1}\end{pmatrix}$.
    
    \item Since $R_0$ has more rows than columns, $R_1$ may be null which means than $R_0\in\St(k-r,l-r)$, that is $R_0^\top R_0=I_{l-r}$. A simple calculus shows than $YR^\top=[Y_r~Y_{l-r}R_0^\top~\mathbf{0}_{n-k}]$. Therefore, the curve $c(t)=(1-t)X+tYR^\top$ has its $n-k$ columns identically null so it has rank less than $k$. But it also has rank at least $k$ because $\rk(\Sigma)=\rk(X)=k$. So $c$ and $\gamma^{R_0}_{\Sigma\to\Lambda}$ are of rank exactly $k$ on $[0,1)$ (and on $[0,1]$ if $l=\rk(\Lambda)=k$).
    
    \item At the other extremity, there is $R_0=0$ (and $R_1\in\St(n-k,l-r)$). For $r=l$, it corresponds to the empty matrix. In this case, $XRY^\top=X_rY_r^\top$. Inspired by the case of the unique geodesic in $\Sym^+(n,k)$ (with $k=l=r$), we notice that $XH^-X^\top YY^\top=XH^-HY^\top=X\Diag(I_r,0)Y^\top=X_rY_r^\top=XRY^\top$. Therefore, denoting $X=UD^{1/2}V^\top$ with $U,V\in\St(n,k)$ and $D\in\Diag^+(k)$, we have:
    \begin{align*}
        XRY^\top&=XH^-X^\top\Lambda\\
        &=UD^{1/2}V^\top((VD^{1/2}U^\top\Lambda UD^{1/2}V^\top)^{1/2})^-VD^{1/2}U^\top\\
        &=UD^{1/2}((D^{1/2}U^\top\Lambda UD^{1/2})^{1/2})^-D^{1/2}U^\top\\
        &=UD^{1/2}U^\top((UD^{1/2}U^\top\Lambda UD^{1/2}U^\top)^{1/2})^-UD^{1/2}U^\top\\
        &=\Sigma^{1/2}((\Sigma^{1/2}\Lambda\Sigma^{1/2})^{1/2})^-\Sigma^{1/2}\Lambda.
    \end{align*}
    Thus the minimizing geodesic writes:
    \begin{equation*}
        \gamma^0_{\Sigma\to\Lambda}(t)=(1-t)^2\Sigma+t^2\Lambda+2t(1-t)\,\sym(\Sigma^{1/2}((\Sigma^{1/2}\Lambda\Sigma^{1/2})^{1/2})^-\Sigma^{1/2}\Lambda).
    \end{equation*}
    When the geodesic is unique, i.e. when $r=l$, i.e. when $R_0$ is the empty matrix, it has rank exactly $k$ on $[0,1)$.
\end{enumerate}
\end{proof}

\end{appendices}

\bibliographystyle{siamplain}
\bibliography{references}

\end{document}